\documentclass{article}
\usepackage[utf8]{inputenc}

\usepackage{ucs}
\usepackage{amsmath}
\usepackage{amsthm}
\usepackage{amsfonts}
\usepackage{amssymb}
\usepackage{comment}
\usepackage{appendix}

\usepackage[margin=1.3in]{geometry}

\usepackage{xfrac}
\usepackage{faktor}

\makeatletter
\usepackage{tikz}

\usetikzlibrary{calc}
\usetikzlibrary{cd}
\usetikzlibrary{shapes}

\makeatother

\newcommand{\C}{\mathcal{C}}

\newtheorem{thm}{Theorem}[section]
\newtheorem*{thm*}{Theorem}
\newtheorem{prop}[thm]{Proposition}
\newtheorem*{prop*}{Proposition}
\newtheorem{lemma}[thm]{Lemma}

\theoremstyle{definition}

\newtheorem{remark}[thm]{Remark}

\newtheorem{ass}[thm]{Assumption}
\newtheorem*{defn*}{Definition}
\newtheorem{notation}[thm]{Notation}

\numberwithin{equation}{section}

\author{\Large{Jim Delitroz and Riccardo W. Maffucci}}
\newcommand{\Addresses}{{
		\footnotesize
		
		R.W.~Maffucci, \textsc{EPFL MA SB,
			Lausanne, Switzerland 1015}\par\nopagebreak\vspace{-0.35cm}
		\texttt{riccardowm@hotmail.com} \textit{(corresponding author)}
		
		J.~Delitroz, \textsc{EPFL MA SB,
			Lausanne, Switzerland 1015}\par\nopagebreak\vspace{-0.35cm}
		\texttt{jim.delitroz@epfl.ch}
}}

\title{\Large{\uppercase{\bf On unigraphic polyhedra with one vertex of degree} ${\bf p-2}$}}

\date{}
\begin{document}
\maketitle
\Addresses

\begin{abstract}
A sequence $\sigma$ of $p$ non-negative integers is unigraphic if it is the degree sequence of exactly one graph, up to isomorphism. A polyhedral graph is a $3$-connected, planar graph. We investigate which sequences are unigraphic with respect to the class of polyhedral graphs, meaning that they admit exactly one realisation as a polyhedron.

We focus on the case of sequences with largest entry $p-2$. We give a classification of polyhedral unigraphic sequences starting with $p-2,p-2$, as well as those starting with $p-2$ and containing exactly one $3$. Moreover, we characterise the unigraphic sequences where a few vertices are of high degree. We conclude with a few other examples of families of unigraphic polyhedra.
\end{abstract}

\tableofcontents
%\newpage
\section{Introduction}

\subsection{Background}

In this paper, we will consider finite graphs without multi-edges, nor loops. We say that a graph $G$ is planar if it can be embedded in the plane in a way such that the edges only intersect at their extremities. Kuratowski showed in 1930 \cite{Kuratowski} that $G$ is planar if and only if it does not contain any subgraph homeomorphic from $K_{3,3}$ or $K_5$. A graph $G$ of order $p$ is $n$-connected if $p>n$ and for any $n-1$ vertices $v_1,...,v_{n-1}$ of $G$, the graph $G-v_1-...-v_{n-1}$ is still connected.

We will focus on the graphs that are planar and 3-connected: they correspond topologically to the convex 3-dimensional polyhedra ($3$-polytopes), in the sense that if two polyhedra are homeomorphic to each other, then their corresponding graphs will be isomorphic, as proved by Rademacher-Steinitz \cite{Steinitz}. In the following, they will be simply referred to as polyhedra. In 1961, Tutte gave an algorithm to construct all polyhedra recursively on the number of edges \cite{tutte1}. Nowadays, this class of graphs is widely studied, as it has applications in various fields, such as chemistry \cite{Rouvray, Sciriha}, maximization problems in revenue management \cite{Kuyumcu}, and the spread of information in social networks \cite{Chen}.

%For $v$ a vertex of a graph $G$, its degree $\text{deg}_G(v)$ is simply the number of vertices adjacent to it, or equivalently, the number of edges that contain it.
Given a graph $G$ on $p$ vertices $v_1,...,v_p$, one can write its degree sequence
\[\sigma=d_1,...,d_p\] by setting $\text{deg}_G(v_i)=d_i$ for $1\leq i\leq p$. We will write degree sequences in a non-increasing way, which can always be done up to vertex relabelling. In a degree sequence, we will sometimes use the notation $d_i^n$ to mean $n$ terms equal to $d_i$, e.g.
\[6,5^3,4^3,3=6,5,5,5,4,4,4,3.\]

A sequence $s$ of integers is called graphical if one can build a graph $G$ whose degree sequence is $\sigma$, and in this case $G$ is called a realization of $\sigma$. Hakimi and Havel showed a necessary and sufficient condition for a sequence to be graphical, and the Havel-Hakimi algorithm \cite{Havel,Hakimi} gives a way to build such realization.

A graphical sequence is called \textbf{unigraphic} if there is, up to isomorphism, only one graph $G$ that realizes $\sigma$. In this case, $G$ is called a unigraph. Unigraphic sequences have been studied since the 1960s \cite{Hakimi 2, Koren, li1975}. There is a classification for unigraphic sequences of non-$2$-connected graphs \cite{john80}, and also for unigraphic sequences satisfying $d_2=d_{p-1}$ \cite[Theorem 6.1]{Koren}.

A natural question is to determine whether a sequence is unigraphic with respects to a certain family of graphs. For example, the sequence 
\[2,2,2,1,1\] admits the two realizations of Figure \ref{fig:22211}, however it is unigraphic with respect to the trees, i.e. if two trees realize $\sigma$, then they are isomorphic (for a characterization of unigraphic trees, see \cite[ex. 6.11]{Harary}).

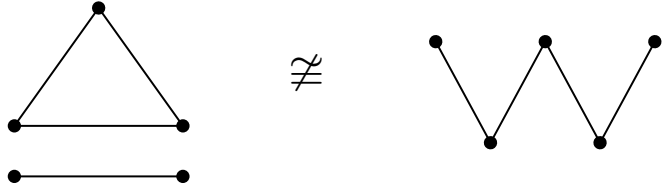
\begin{figure}
\centering    
\tikzset{every picture/.style={line width=0.75pt}} %set default line width to 0.75pt        

\begin{tikzpicture}[x=0.75pt,y=0.75pt,yscale=-0.85,xscale=0.85]
%uncomment if require: \path (0,381); %set diagram left start at 0, and has height of 381

%Straight Lines [id:da724172339742229] 
\draw    (50,200) -- (150,200) ;
\draw [shift={(150,200)}, rotate = 0] [color={rgb, 255:red, 0; green, 0; blue, 0 }  ][fill={rgb, 255:red, 0; green, 0; blue, 0 }  ][line width=0.75]      (0, 0) circle [x radius= 3.35, y radius= 3.35]   ;
\draw [shift={(50,200)}, rotate = 0] [color={rgb, 255:red, 0; green, 0; blue, 0 }  ][fill={rgb, 255:red, 0; green, 0; blue, 0 }  ][line width=0.75]      (0, 0) circle [x radius= 3.35, y radius= 3.35]   ;
%Straight Lines [id:da9457593871808072] 
\draw    (100,130) -- (150,200) ;
\draw [shift={(100,130)}, rotate = 54.46] [color={rgb, 255:red, 0; green, 0; blue, 0 }  ][fill={rgb, 255:red, 0; green, 0; blue, 0 }  ][line width=0.75]      (0, 0) circle [x radius= 3.35, y radius= 3.35]   ;
%Straight Lines [id:da332558422347754] 
\draw    (50,200) -- (100,130) ;
%Straight Lines [id:da8242579295829418] 
\draw    (50,230) -- (150,230) ;
\draw [shift={(150,230)}, rotate = 0] [color={rgb, 255:red, 0; green, 0; blue, 0 }  ][fill={rgb, 255:red, 0; green, 0; blue, 0 }  ][line width=0.75]      (0, 0) circle [x radius= 3.35, y radius= 3.35]   ;
\draw [shift={(50,230)}, rotate = 0] [color={rgb, 255:red, 0; green, 0; blue, 0 }  ][fill={rgb, 255:red, 0; green, 0; blue, 0 }  ][line width=0.75]      (0, 0) circle [x radius= 3.35, y radius= 3.35]   ;
%Straight Lines [id:da15899336990464152] 
\draw    (300,150) -- (332.5,210) ;
\draw [shift={(332.5,210)}, rotate = 61.56] [color={rgb, 255:red, 0; green, 0; blue, 0 }  ][fill={rgb, 255:red, 0; green, 0; blue, 0 }  ][line width=0.75]      (0, 0) circle [x radius= 3.35, y radius= 3.35]   ;
\draw [shift={(300,150)}, rotate = 61.56] [color={rgb, 255:red, 0; green, 0; blue, 0 }  ][fill={rgb, 255:red, 0; green, 0; blue, 0 }  ][line width=0.75]      (0, 0) circle [x radius= 3.35, y radius= 3.35]   ;
%Straight Lines [id:da19555793418581557] 
\draw    (332.5,210) -- (365,150) ;
\draw [shift={(365,150)}, rotate = 298.44] [color={rgb, 255:red, 0; green, 0; blue, 0 }  ][fill={rgb, 255:red, 0; green, 0; blue, 0 }  ][line width=0.75]      (0, 0) circle [x radius= 3.35, y radius= 3.35]   ;
%Straight Lines [id:da6940535740492393] 
\draw    (365,150) -- (397.5,210) ;
\draw [shift={(397.5,210)}, rotate = 61.56] [color={rgb, 255:red, 0; green, 0; blue, 0 }  ][fill={rgb, 255:red, 0; green, 0; blue, 0 }  ][line width=0.75]      (0, 0) circle [x radius= 3.35, y radius= 3.35]   ;
%Straight Lines [id:da030650738700934266] 
\draw    (397.5,210) -- (430,150) ;
\draw [shift={(430,150)}, rotate = 298.44] [color={rgb, 255:red, 0; green, 0; blue, 0 }  ][fill={rgb, 255:red, 0; green, 0; blue, 0 }  ][line width=0.75]      (0, 0) circle [x radius= 3.35, y radius= 3.35]   ;

% Text Node
\draw (211,152) node [anchor=north west][inner sep=0.75pt]  [font=\LARGE] [align=left] {$\displaystyle \ncong $};

\end{tikzpicture}
\caption{Two realisations of $22211$.}
\label{fig:22211}
\end{figure}

\paragraph{Motivation.}
This paper is about sequences that are unigraphic with respect to the class of polyhedral graphs, following the work of \cite{Maffucci 1, Maffucci 2}. When constructing polyhedra with Tutte's algorithm, certain patterns emerge (cf. \cite{gasmaf}). Motivated by data patterns, in \cite{Maffucci 1, Maffucci 2} the second author partially classified the unigraphic polyhedra of graph radius one. A graph $G$ on $p$ vertices has radius one if and only if one vertex $v$ is adjacent to all others, i.e., if and only if the first term of its degree sequence $\sigma$ is $p-1$. This article investigates the follow-up question of unigraphic degree sequences starting with $p-2$. Henceforth, when a sequence is referred to as unigraphic, it will implicitly mean that it is unigraphic with respect to the class of polyhedra.% together with some other topics about unigraphic polyhedra, such as a proof that the only unigraphic 4-regular polyhedra are of order 6,8 and 9.

\subsection{Main results}
We will use Greek letters to denote degree sequences of order $p$, starting with $p-2$. When computing unigraphic sequences of polyhedra via Tutte's algorithm, the following patterns emerge. If $p\geq 7$, and the sequence starts with $p-2,p-2$, then it is of one of two types $\alpha(p),\beta(p)$, with seven exceptions. Similarly, if $p\geq 7$ and the sequence starts with $p-2$ and contains exactly one $3$, then it is of one of two types $\alpha(p),\gamma(p)$, with one exception. Our first two theorems rigorously prove these assertions. % An apex may be inserted to specify the sequence's length, so that $\sigma=\sigma^n$ is the degree sequence of a polyhedron of order $n$.

\begin{thm}\label{thm2}
	Let $p\geq 7$, and
	\begin{align*}
	\alpha=\alpha(p) &= p-2,p-2,5, 4^{p-4},3, \\
	\beta=\beta(p) &= p-2,p-2, 4^{p-2}.\\
	\end{align*}
	Then $\alpha$ and $\beta$ are polyhedral unigraphic sequences. Moreover, if $p\geq 11$, then there are no other polyhedral unigraphic sequences starting with $p-2,p-2$.
\end{thm}
The proof of Theorem \ref{thm2} may be found in section \ref{sec2}. For $p\leq 10$, we have computed the data in Table \ref{tab0}.
\begin{table}[h!]
	\centering
	%\begin{equation}
	$\begin{array}{|l|l|l|}
	\hline \text{Order}&\text{Unigraphic sequences beginning with } p-2,p-2\\	\hline
	p=6& 4,4,3,3,3,3; \hspace{1cm} 4,4,4,4,3,3; \hspace{1cm} 4,4,4,4,4,4.\\\hline  
	p=7& 5,5,5,4,4,4,3; \hspace{1cm}
	5,5,4,4,4,4,4.\\\hline
	p=8& 6, 6, 5, 4, 4, 4, 4, 3; \hspace{1cm} 6, 6, 4, 4, 4, 4, 4, 4; \hspace{1cm} 6, 6, 6, 6, 3, 3, 3, 3; \\&6, 6, 6, 5, 4, 3, 3, 3. \\\hline
	p=9&  7, 7, 5, 4, 4, 4, 4, 4, 3; \hspace{1cm} 7, 7, 4, 4, 4, 4, 4, 4, 4; \hspace{1cm} 7, 7, 6, 6, 4, 3, 3, 3, 3; \\&7, 7, 6, 5, 5, 3, 3, 3, 3; \hspace{1cm} 7, 7, 5, 5, 5, 4, 3, 3, 3.\\\hline
	p=10&8, 8, 5, 4, 4, 4, 4, 4, 4, 3; \hspace{1cm} 8, 8, 4, 4, 4, 4, 4, 4, 4, 4; \hspace{1cm} 8, 8, 6, 5, 5, 4, 3, 3, 3, 3;\\&
	8, 8, 5, 5, 5, 5, 3, 3, 3, 3.
	\\\hline
	\end{array}$
	\caption{Unigraphic sequences starting with $p-2,p-2$, for $p\leq 10$.}
	\label{tab0}
\end{table}

In particular, for $p\geq 7$, there are seven exceptional sequences (i.e., not of types $\alpha$ or $\beta$), of which two each for $p=8$ and $p=10$, and three for $p=9$. Note that $\beta(p)$ is the sequence of the $p-2$-gonal bipyramid.

\begin{thm}\label{thm1}
Let $p\geq 7$. The only polyhedral unigraphic sequences of order $p$ starting with $p-2$ and containing exactly one $3$ are the following:

\begin{align*}
  \alpha=\alpha(p) &= p-2,p-2,5, 4^{p-4},3, \\
  \gamma=\gamma(p) &= p-2, 4^{p-2}, 3 \qquad p \text{ odd},
 \end{align*}
 together with the exceptional $6,5^3,4^3,3$.
\end{thm}
Theorem \ref{thm1} will be proven in section \ref{sec3}. The following result characterises unigraphic sequences where a few vertices are of high degree.% These two results show the existence of some families of polyhedral graphs, as well as their uniqueness. No other sequence respecting the constraints will be unigraphic.

\begin{thm}
	\label{thm3}
Let $p\geq 11$, and $\nu(p)$ be a non-increasing polyhedral sequence beginning with
\[p-2,d_1,d_2,d_3,\dots,\]
where
\[d_1,d_2,d_3\geq 7\]
and
\[d_1+d_2+d_3\geq p+10.\]
If $\nu(p)$ is unigraphic, then \[\nu_m=\nu_m(p)=p-2, (m+6)^2,p-2m-2,4^{p-8},3^4,\]
with $1\leq m<\frac{1}{2}(p-8)$.
\end{thm}

Determining all unigraphic sequences starting with $p-2$, or even more generally, all the unigraphic with no constraints, seems to be a difficult problem, that could be the focus of future work. As another step in this direction, here we also give another family of unigraphic sequences starting with $p-2$ not included in the above results.

\begin{prop}\label{twoseq}
For even $p\geq 16$, the sequence
\[\mu=\mu(p)=p-2,\frac{p-2}{2}, 6^{(p-2)/2},3^{(p-2)/2}\] 
is unigraphic.
\end{prop}
The proof of Proposition \ref{twoseq} is relegated to Appendix \ref{appa}.

\paragraph{Data availability statement.}
The data produced to motivate this work is available on request.

\paragraph{Acknowledgements.}
J.D. worked on this project as partial fulfilment of his master thesis at EPFL, autumn 2022, under the supervision of Maryna Viazovska and R.M.
\\
R.M. was supported by Swiss National Science Foundation project 200021\_184927, held by Prof. Maryna Viazovska.

\section{Proof of Theorem \ref{thm2}}
\label{sec2}
\subsection{All polyhedral sequences starting with $p-2,p-2$}
In this section, we classify the polyhedral degree sequences starting with $p-2,p-2$. Among them, we will find in particular the unigraphic ones, to prove Theorem \ref{thm2}. In the next statement, in a degree sequence we write `$\dots$' as a shorthand for $p-2,p-2$ followed by as many $4$'s as necessary so that the sequence has length $p$. For instance, for fixed $p$, the expression $\sigma_1=6,6,3,3,3,3,\dots$ is a shorthand for
\[\sigma_1(p)=p-2,p-2,6,6,4^{p-8},3,3,3,3.\]

\begin{prop}
	\label{theprop}
	Let $p\geq 7$, and $G$ be a polyhedral graph on $p$ vertices, at least two of which of degree $p-2$. Then the possible degree sequences of $G$ are listed in Table \ref{tab2}.
	\begin{table}[h!]
		\centering
		$\begin{array}{|l|l|l|}
		\hline
		\sigma_1=6,6,3,3,3,3,\dots
		&\sigma_2=6,5,5,3,3,3,3,\dots
		&\sigma_3=6,5,3,3,3,\dots
		\\\hline\sigma_4=5,5,5,5,3,3,3,3,\dots
		&\sigma_5=5,5,5,3,3,3,\dots
		&\sigma_6=5,5,3,3,\dots
		\\\hline\sigma_7=5,5,3,3,3,3,\dots
		&\sigma_8=5,3,\dots =\alpha
		&\sigma_9=5,3,3,3,\dots
		\\\hline\sigma_{10}=\beta
		&\sigma_{11}=3,3,\dots
		&\sigma_{12}=3,3,3,3,\dots.
		\\\hline
		\end{array}$
		\caption{All polyhedral sequences starting with $p-2,p-2$, for $p\geq 11$. The dots in each are a shorthand for $p-2,p-2$ followed by as many $4$'s as necessary so that each sequence has length $p$. The sequence $\beta$ is the same as in the statement of Theorem \ref{thm2}, i.e. the sequence of a $p-2$-gonal bipyramid.}
		\label{tab2}
	\end{table}
\end{prop}
The present section is dedicated to proving Proposition \ref{theprop} and Theorem \ref{thm2}. Note that in $\sigma_1,\sigma_5,\sigma_7$ we have $p\geq 8$, in $\sigma_2$ we have $p\geq 9$, and in $\sigma_4$ we have $p\geq 10$. In all other cases in Table \ref{tab2}, the sequence is defined for $p\geq 7$.

%We employ the following notation.
%\begin{itemize}
	%\item
In this section we will denote by $x,y$ two distinct vertices of degree $p-2$ in the polyhedron $G$, and by $H:=G-x-y$ the connected graph obtained by removing $x,y$ from $G$. The first distinction to make is whether $x$ and $y$ are adjacent.
	%\item
	%Degree sequences will be written as $\sigma_i$ (reserved for candidate unigraphic) and $\tau_j$, with $i,j$ natural numbers.
%\end{itemize}

\subsection{$x$ and $y$ are not adjacent}\label{x y not adj section}
\label{sec:ca1}

In this case, we know that every vertex $h$ of $H$ is adjacent to both $x$ and $y$ in $G$.

\begin{lemma}
	If $xy\not\in E(G)$, then $H$ is either a path or a cycle.
\end{lemma}
\begin{proof}
	Let $h_1,h_2,...,h_{p-2}$ be the vertices of $H$. By connectivity, none of them are isolated in $H$. Moreover, we have 
	\begin{equation}
	\label{eqn:rel}
	\deg_H(h_i)\leq 2 \qquad 1\leq i\leq p-2
	\end{equation}
	due to planarity. Indeed, if $h_1$ were adjacent to $h_2,h_3,h_4$, say, then $G$ would contain a copy of $K_{3,3}$, with partition  $\{x,y,h_1\}$ and $\{h_2,h_3,h_4\}$, contradicting Kuratowski's Theorem. It now follows that $H$ has either $p-3$ or $p-2$ edges, the lower bound coming from connectivity and the upper bound from \eqref{eqn:rel}. If the size of $H$ is $p-3$ then $H$ is simply a path. If instead the size of $H$ is $p-2$, then $H$ contains a cycle, and in fact by \eqref{eqn:rel} $H$ is a cycle.
\end{proof}

To complete the analysis of the case $xy\not\in E(G)$, 
if $H$ is a path then the degree sequence of $G$ is $\sigma_{11}$ from Table \ref{tab2}.
A priori, $\sigma_{11}$ is a unigraphic candidate, however, later we will build a graph of same degree sequence, but where $x$ and $y$ are adjacent: this will suffice to prove that $\sigma_{11}$ is not unigraphic. If $H$ is a cycle, then $G$ is simply a $(p-2)$-gonal bipyramid, of sequence $\beta$.

\subsection{$x$ and $y$ are adjacent}
\label{sec:ca2}
%\begin{notation}
Let $a$ be the only vertex not adjacent to $y$, and $b$ the only vertex not adjacent to $x$. Let $F=G-y$, and $W=V(F)\setminus \{a\}$. In $G$, $y$ is adjacent to every vertex of $W$, therefore by planarity, these are all on the boundary of the same face in $F$. In other words, the elements of $W$ lie on a cycle $C$ of $F$. Similarly, one defines $F'=G-x$, and $W'=V(F')\setminus \{b\}$, to see that the elements of $W'$ lie on a cycle $C'$ of $F'$. Using these notations, we define the following assumption.

\begin{ass}
	\label{as}
	In $G$ these exists a cycle $C$ containing every element of $W$ but not $a$, and moreover there exists a cycle $C'$ containing every element of $W'$ but not $b$.
\end{ass}

We proceed to analyse in detail the two cases of this assumption being verified or not.

%Recall that two vertices are \emph{neighbours} if they are adjacent in $C$, and \emph{linked} if they are adjacent in $G$, but not neighbours. We will denote the two neighbours of $x$ by $v_1$ and $v_2$.

\subsubsection{Assumption \ref{as} is verified}
\label{sec:asve}
We utilise the vertex labelling
\[V(G)=\{x,y,a,b,v_1,v_2,\dots,v_{p-4}\}.\]
To begin the construction of $G$, we draw the cycle $C$ containing for the moment the vertices
\[x,v_1,v_2,\dots,v_{p-4}\]
in order, and add the edges $xy$,
\[yv_i, \quad 1\leq i\leq p-4,\]
and
\[xv_i, \quad 2\leq i\leq p-5.\]
The resulting (planar) graph, of sequence $p-3,p-3,4^{p-6},3,3$, will be denoted by $G'$. It is sketched in Figure \ref{fig:XYZ}.

\begin{figure}[h!]
	\centering
	\includegraphics[width=7cm]{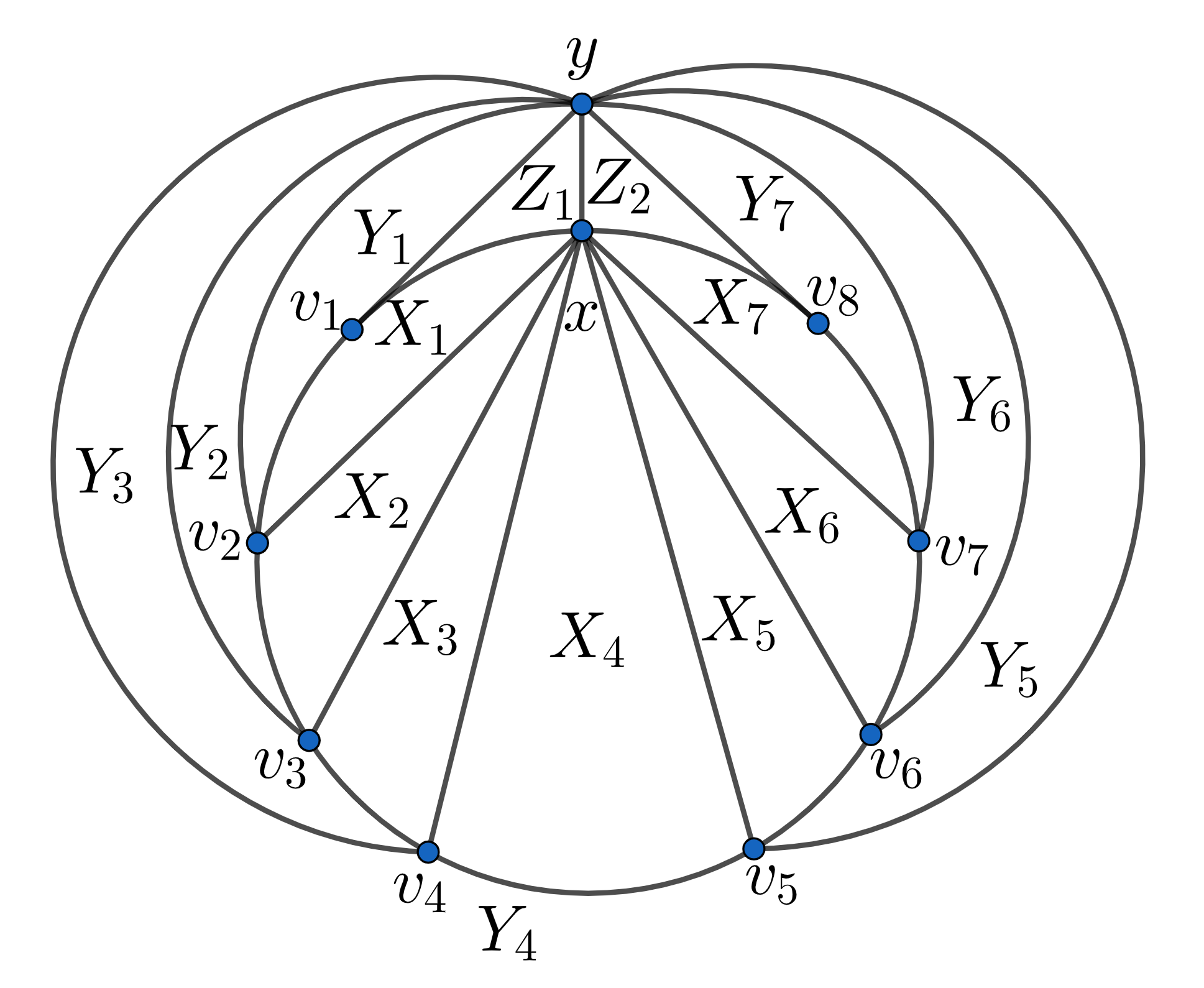}
	\caption{The subgraph $G'$ of $G$, for $p=12$.}
	\label{fig:XYZ}
\end{figure}

To recover $G$, it remains to insert $a,b$ and their incident edges. The notation $[u_1,u_2,\dots,u_n]$ indicates an $n$-gonal region of a planar graph, bounded by the cycle $u_1,u_2,\dots,u_n$. The regions of $G'$ are all triangular, and we will write
\[
Z_1=[x,y,v_1], \qquad
Z_2=[x,y,v_{p-4}]
\]
and, for $1\leq i\leq p-5$,
\[
X_i=[x,v_i,v_{i+1}], \qquad Y_i=[y,v_i,v_{i+1}].
\]

If $a,b$ are adjacent, then they must lie inside the same region of $G'$, and moreover, this region must be one of the two with boundary including the edge $xy$ (either $Z_1$ or $Z_2$). Up to isomorphism,  %(i.e. using the symmetry between $x$ and $y$, and correspondingly between $a$ and $b$),
there is thus only one choice, i.e. this region is the triangle $Z_1$. Since $a,b$ have degree at least three, we obtain a unique graph in this case, of sequence $\sigma_9$.

For the rest of section \ref{sec:asve}, we will assume that $ab\not\in E(G)$. Using planarity, together with the fact that $a,b$ are of degree at least three, we deduce that $a$ lies inside a triangle $X_i$ of $G'$, and $b$ inside $Y_j$, where $1\leq i,j\leq p-5$. There are two cases.
\begin{enumerate}
	\item
	At least one of $i,j$ equals $1$, say $i=1$. There are a few subcases, and the resulting sequences are listed in Table \ref{tab3}. Due to symmetries, this also covers the case when at least one of $i,j$ equals $p-5$.

	\begin{table}[h!]
		\centering
		$\begin{array}{|l|l|l|}
		\hline
		\text{Subcase}&\text{Sequences}
		\\\hline
		i=j=1
		&\sigma_3
		\\\hline
		i=1, \ j=2
		&\sigma_3
		\\\hline
		i=1, \ j=p-5
		&\sigma_6
		\\\hline
		i=1, \ j\neq 1,2,p-5
		&\sigma_5
		\\\hline
		\end{array}$
		\caption{Assumption \ref{as} holds, $a$ lies inside $X_1$, and $b$ inside $Y_j$, for some $1\leq j\leq p-5$.}
		\label{tab3}
	\end{table}

	\begin{remark}
		\label{rem:1}
		For $j\neq 1,2,p-5$ to make sense, we must have $p\geq 8$. We find the exceptional sequences
		\[\sigma_5(8)=6, 6, 6, 5, 4, 3, 3, 3 \qquad \sigma_5(9)=7, 7, 5, 5, 5, 4, 3, 3, 3\]
		listed in Table \ref{tab0}. We record that if $p\geq 10$, then  $\sigma_5(p)$ is not unigraphic.
	\end{remark}
	
	\item
	$i,j\neq 1,p-5$. For this to make sense, we must have $p\geq 8$. As above, we fix $i$ and collect the resulting subcases in Table \ref{tab4}.
	
	%We record that for $p\geq 11$, there are here enough choices for $i,j$ to produce for each of $\sigma_1,\sigma_2,\sigma_4$ at least two non-isomorphic polyhedra of same sequence, hence none of $\sigma_1,\sigma_2,\sigma_4$ are unigraphic.
	\begin{table}[h!]
		\centering
		$\begin{array}{|l|l|l|}
		\hline
		\text{Subcase}&\text{Sequences}
		\\\hline
		j=i
		&\sigma_1
		\\\hline
		|i-j|=1
		&\sigma_2
		\\\hline
		|i-j|\geq 2
		&\sigma_4
		\\\hline
		\end{array}$
		\caption{Assumption \ref{as} holds, $a$ lies inside $X_i$, and $b$ inside $Y_j$, for some $i,j\neq 1,p-5$.}
		\label{tab4}
	\end{table}

\begin{remark}
	\label{rem:2}
	Similarly to Remark \ref{rem:1}, we have
	\begin{equation}
	\label{eqn:exc}
	\sigma_1(8)=6,6,6,6,3,3,3,3, \qquad \sigma_1(9)=7,7,6,6,4,3,3,3,3.
	\end{equation}
	These are two of the exceptional entries in Table \ref{tab0}. For $p\geq 10$, there are here enough choices for $i,j$ to produce at least two non-isomorphic polyhedra of sequence $\sigma_1(p)$.
	\\
	Similarly, $|i-j|\geq 1$ implies $p\geq 9$, and we find \[\sigma_2(9)=7, 7, 6, 5, 5, 3, 3, 3, 3, \qquad \sigma_2(10)=8, 8, 6, 5, 5, 4, 3, 3, 3, 3,\] and if $p\geq 11$, then $\sigma_2(p)$ is not unigraphic.
	\\
	If $|i-j|\geq 2$ then $p\geq 10$, and we find \[\sigma_4(10)=8,8,5,5,5,5,3,3,3,3,\] and if $p\geq 11$, then $\sigma_4(p)$ is not unigraphic.
\end{remark}
\end{enumerate}

\subsubsection{Assumption \ref{as} is not verified}
Here either $b$ may be inserted into $G'$ via an elementary subdivision of the edge $v_iv_{i+1}$, for some $1\leq i\leq p-5$, or $a$ may be inserted into $G'$ via an elementary subdivision of the edge $v_jv_{j+1}$, for some $1\leq j\leq p-5$, or both. With no loss of generality, assume the former. Note that we can also suppose that $v_iv_{i+1}\not\in E(G)$, otherwise we would be in the already studied case of $b$ lying in a triangle of $G'$. There are three cases.
\begin{enumerate}
	\item
	$a,b$ lie on the cycle $C$ between the same $v_i$ and $v_{i+1}$. The order in which $a,b$ appear in $C$ is irrelevant by symmetry, so let's say that
	\[v_ia, \ ab, \ bv_{i+1}\in E(G).\]
	We summarise the subcases in Table \ref{tab5}.
	\begin{table}[h!]
		\centering
		$\begin{array}{|l|l|l|}
		\hline
		\text{}&i=1,p-5&i\neq 1,p-5
		\\\hline
		av_{i+1}, \ bv_{i}\not\in E(G)
		&\sigma_{12}&\sigma_{12}
		\\\hline
		\text{Exactly one of }av_{i+1}\ bv_{i} \text{ is an edge in } G
		&\sigma_{11}&\sigma_{9}
		\\\hline
		av_{i+1}, \ bv_{i}\in E(G)
		&\alpha&\sigma_6
		\\\hline
		\end{array}$
		\caption{Assumption \ref{as} does not hold, $a,b$ lie on the cycle $C$ between the same $v_i$ and $v_{i+1}$, for some $1\leq i\leq p-5$.}
		\label{tab5}
	\end{table}
	%In the second subcase, we get $\sigma_{11}$ or $\sigma_{9}$ depending on whether $i=1,p-5$ or not, and similarly in the third subcase.
	An example may be found in Figure \ref{fig:s9}.
	\begin{figure}[h!]
		\centering
		\includegraphics[width=7cm]{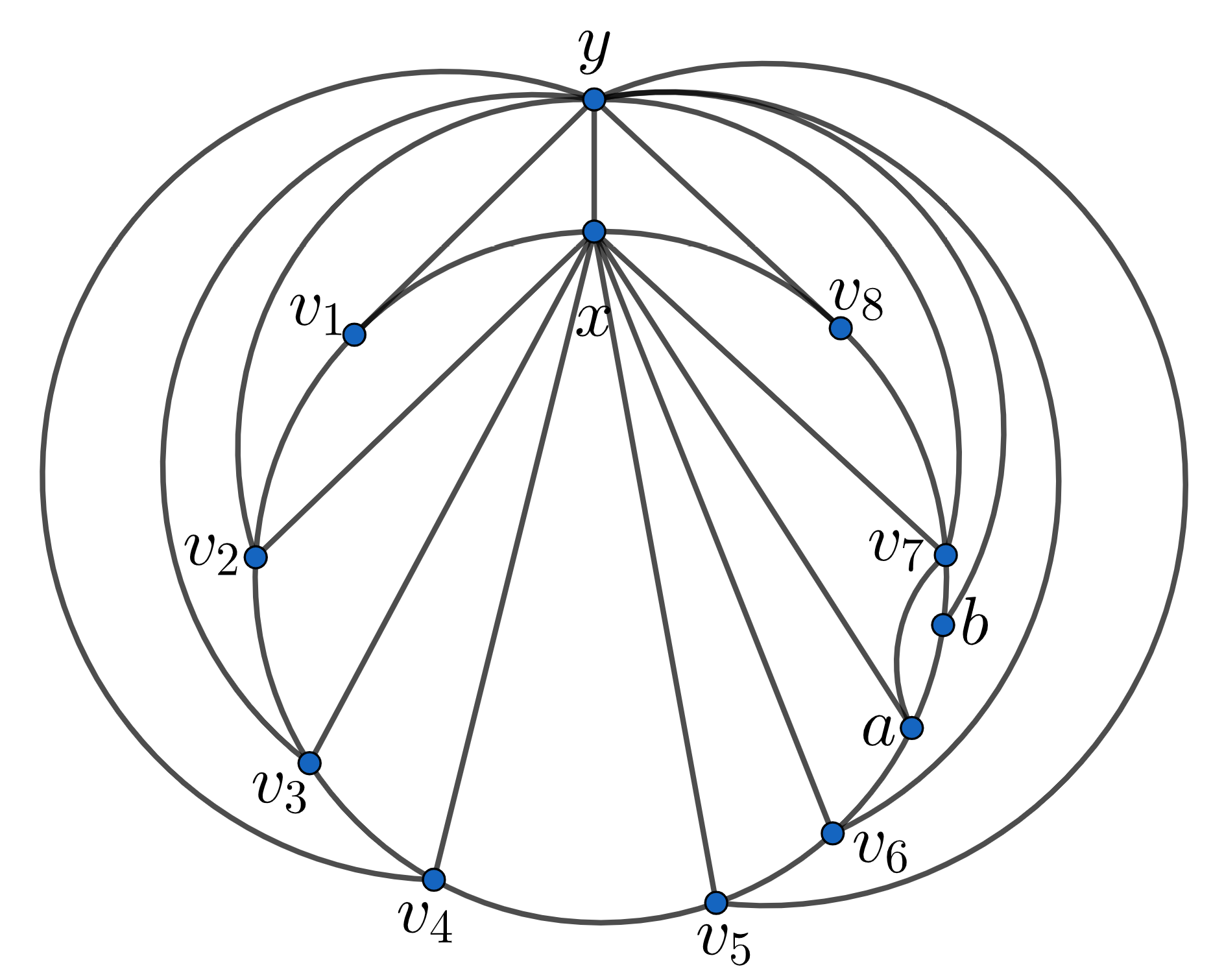}
		\caption{For $p=12$, the case of Assumption \ref{as} not verified, $a,b$ lying on $C$ between $v_6$ and $v_7$, $av_7\in E(G)$, and $bv_6\not\in E(G)$. The resulting sequence is $\sigma_9(12)=10,10,5,4^6,3,3,3$.}
		\label{fig:s9}
	\end{figure}
	\item
	$b$ lies on the cycle $C$ between $v_i$ and $v_{i+1}$, and $a$ lies on the cycle $C$ between $v_j$ and $v_{j+1}$, with $i\neq j$. Here necessarily $\deg_G{a}=\deg_G{b}=3$, and we obtain $\sigma_{12}$.
	\item
	$b$ lies on the cycle $C$ between $v_i$ and $v_{i+1}$, and $a$ lies inside the triangle $X_j$ of $G'$, for some $1\leq i,j\leq p-5$. It is important to remark that necessarily $ab\not\in E(G)$ even if $i=j$. Indeed, if $i=j$ and $ab\in E(G)$, then either all of $v_i b, \ ba, \ av_{i+1}$ are edges, or all of $v_i a, \ ab, \ bv_{i+1}$ are edges. Thereby, in any case, by redefining $C$ we fall back into the already inspected scenario of $a,b$ both lying on the cycle between $v_i,v_{i+1}$. Here we obtain either $\sigma_9$ or $\sigma_7$, depending on whether $j=1,p-5$ or not.
	\begin{remark}
		\label{rem:3}
	In case $j\neq 1,p-5$, we must have $p\geq 8$. The sequence $\sigma_7(p)$ is not unigraphic for any value of $p$. For instance, if $p=8$, the choice of $j=2$ is unique, but the different choices of $i=1,2$ produce non-isomorphic graphs of sequence $\sigma_7(8)$.
	\end{remark}
	%\begin{table}[h!]
	%	\centering
	%	$\begin{array}{|l|l|l|}
	%	\hline
	%	&i=1,p-5&i\neq 1,p-5
	%	\\\hline
	%	i=j
	%	&\sigma_{9}&\sigma_{7}
	%	\\\hline
	%	i\neq j
	%	&\alpha&\sigma_6
	%	\\\hline
	%	\end{array}$
	%	\caption{.}
	%	\label{tab6}
	%\end{table}
	%As an important remark, we have obtained $\alpha$ a second time, however the two cases that produce $\alpha$ are not mutually exclusive. Indeed, when $b$ lies on the cycle $C$ between $v_i$ and $v_{i+1}$, and $a$ lies inside the triangle $X_i$ of $G'$, in the subcase $i=1$, $ab\in E(G)$, 
	%\item
	%Lastly, $b$ lies on the cycle $C$ between $v_i$ and $v_{i+1}$, and $a$ lies inside a triangle $X_j$ of $G'$, with $i\neq j$. Here we obtain either $\sigma_9$ or $\sigma_7$, depending on whether $j=1,p-5$ or not.
\end{enumerate}
Having studied all cases, Proposition \ref{theprop} is hence proven. We now turn to completing the proof of Theorem \ref{thm2}.

\subsection{Completing the Proof of Theorem \ref{thm2}}
We inspect which of the sequences listed in Proposition \ref{theprop} appear more than once throughout the various cases in sections \ref{sec:ca1} and \ref{sec:ca2}. Each of $\sigma_3,\sigma_6,\sigma_9,\sigma_{11},\sigma_{12}$ appears at least twice in mutually exclusive cases in the construction of $G$.% One need a little care, since not all of the cases and subcases are mutually exclusive. In fact, $\alpha$ does appear twice, however once checks that the two polyhedra obtained are isomorphic. The corresponding cases where we constructed $\alpha$ turn out not to be mutually exclusive. Indeed, 

As for $\sigma_1,\sigma_2,\sigma_4,\sigma_5,\sigma_7$, we have already noted in Remarks \ref{rem:1}, \ref{rem:2}, and \ref{rem:3} that these collectively produce the seven exceptional entries in Table \ref{tab0} for $p\geq 7$ (two each for $p=8$ and $p=10$, and three for $p=9$). We had also recorded that apart from these seven exceptions, $\sigma_1,\sigma_2,\sigma_4,\sigma_5,\sigma_7$ are not unigraphic.

On the other hand, the sequences $\sigma_8=\alpha$ and $\sigma_{10}=\beta$ are unigraphic, since each appears only once in the cases for the construction of $G$, and since all scenarios for sequences starting with $p-2,p-2$ were analysed while proving Proposition \ref{theprop}. The proof of Theorem \ref{thm2} is complete.

\section{Proof of Theorem \ref{thm1}}
\label{sec3}
In this section we will use the following notation.
%\begin{notation}\label{main notation}
\begin{itemize}
	\item
	The letter $G$ will denote a polyhedron, $y$ a vertex of degree $p-2$, $a$ the only vertex not adjacent to $y$, and $c$ the only vertex of degree $3$.
	\item
	We also set $F:=G-y$, the 2-connected planar graph obtained by removing $y$ from $G$. %Except in last proof, 
	Every figure will be a sketch of $F$, as $F$ uniquely determines $G$. We denote by $V$ the vertex set of $F$, and $W:=V\backslash \{a\}$.
\end{itemize}
%\end{notation}
\subsection{Setup}
%The first part of the proof of Theorem \ref{thm1} will develop some tools and techniques that will be used in most of this article, such as Notation \ref{main notation} and Proposition \ref{big cycle}.

\begin{lemma}\label{bigcycle}
	There exists in $F$ one cycle $C$ containing every vertex of $W$, and not containing $a$. Moreover, $a\neq c$.
\end{lemma}
\begin{proof}
In $G$, $y$ is adjacent to every vertex of $W$, therefore by planarity, these are all on the boundary of the same face in $F$. Now by contradiction let $a\in C$. The graph $F$ is then outerplanar, i.e. planar with all vertices lying on the boundary of one region. But then there would be at least two vertices of degree $3$ in $G$, contradiction.

To prove the last statement, assume by contradiction that $a=c$. Recall that $p\geq 7$.  %otherwise $G$ is simply the triangular bipyramid (and contains vertices of degree $p-1$).
Letting $v_1,v_2,v_3$ be the three vertices adjacent to $a$, we note that $F-v_3-a+v_1v_2$ is outerplanar, and thus has a vertex $v'$ of degree $2$ other than $v_1$ and $v_2$ (since $v_1$ and $v_2$ are adjacent in this graph). But then $\deg_G(v')=3$, contradiction.
\end{proof}

In particular,
\[\deg_G(a)\geq 4.\]

Let
\[
R:=F-E(C)-c=G-y-E(C)-c,
\]
where $E(C)$ is the edge set of the cycle $C$. Note that it suffices to study $R$, as this graph uniquely determines $F$ and hence the polyhedron $G$.

\begin{lemma}\label{max 1 triangle}
	Either $R$ is a forest, or $R$ contains exactly one cycle, which is a triangle containing the vertex $a$.
\end{lemma}
\begin{proof}
	Let's suppose for contradiction that $R$ contains an $n$-gon with $n\geq 4$. Then at most two edges of this $n$-gon contain $a$, thus there are two consecutive edges $e_1,e_2$ that do not contain this vertex. By the arguments in Lemma \ref{bigcycle}, there are two vertices of degree $3$ in $G$, contradiction. Similarly, if $R$ contains two triangles, then $a$ belongs to both, so that again by Lemma \ref{bigcycle} (and planarity), we get more than one degree $3$ vertex in $G$.
\end{proof}

\subsection{First case: $R$ is a forest}
\label{sec:for}
We call $A$ the connected component of $a$ in $R$, and $P_1,...,P_k$ the other components. Now every forest degree sequence admits a realisation as a disjoint union of a caterpillar and a certain number of copies of $K_2$, cf. \cite{Maffucci 1}. Since $\deg_R(a)=\deg_G(a)\geq 4$, it follows that $A$ is a caterpillar.

\begin{notation}
	We write 
	\[
	\C(j_1,...,j_l)
	\]
	for the caterpillar with non-leaf vertices of degrees $j_1,...,j_l$ in order along its central path, where every $j_i\geq 2$.
	For $n\geq 1$, we write $S_n$ for the star on $n+1$ vertices, so that $S_n=\C(n)$ for every $n\geq 2$. The graph $K_1$ is the trivial caterpillar, and we use the convention $\C(\emptyset)=S_1=K_2$.
\end{notation}
%\begin{example}

For instance, the caterpillar $\C(3,4,3,4,2)$ is depicted in Figure \ref{fig:cat}.

\tikzset{every picture/.style={line width=0.75pt}} %set default line width to 0.75pt        
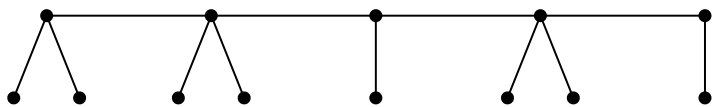
\begin{figure}[h!]
	\centering
	\begin{tikzpicture}[x=0.75pt,y=0.75pt,yscale=-0.83,xscale=0.83]
	%uncomment if require: \path (0,381); %set diagram left start at 0, and has height of 381
	
	%Straight Lines [id:da29569836999593235] 
	\draw    (50,150) -- (150,150) ;
	\draw [shift={(150,150)}, rotate = 0] [color={rgb, 255:red, 0; green, 0; blue, 0 }  ][fill={rgb, 255:red, 0; green, 0; blue, 0 }  ][line width=0.75]      (0, 0) circle [x radius= 3.35, y radius= 3.35]   ;
	\draw [shift={(50,150)}, rotate = 0] [color={rgb, 255:red, 0; green, 0; blue, 0 }  ][fill={rgb, 255:red, 0; green, 0; blue, 0 }  ][line width=0.75]      (0, 0) circle [x radius= 3.35, y radius= 3.35]   ;
	%Straight Lines [id:da5349814054294537] 
	\draw    (150,150) -- (250,150) ;
	\draw [shift={(250,150)}, rotate = 0] [color={rgb, 255:red, 0; green, 0; blue, 0 }  ][fill={rgb, 255:red, 0; green, 0; blue, 0 }  ][line width=0.75]      (0, 0) circle [x radius= 3.35, y radius= 3.35]   ;
	%Straight Lines [id:da6075614283521278] 
	\draw    (250,150) -- (350,150) ;
	\draw [shift={(350,150)}, rotate = 0] [color={rgb, 255:red, 0; green, 0; blue, 0 }  ][fill={rgb, 255:red, 0; green, 0; blue, 0 }  ][line width=0.75]      (0, 0) circle [x radius= 3.35, y radius= 3.35]   ;
	%Straight Lines [id:da3977696792402339] 
	\draw    (350,150) -- (450,150) ;
	\draw [shift={(450,150)}, rotate = 0] [color={rgb, 255:red, 0; green, 0; blue, 0 }  ][fill={rgb, 255:red, 0; green, 0; blue, 0 }  ][line width=0.75]      (0, 0) circle [x radius= 3.35, y radius= 3.35]   ;
	%Straight Lines [id:da8670687278700839] 
	\draw    (50,150) -- (30,200) ;
	\draw [shift={(30,200)}, rotate = 111.8] [color={rgb, 255:red, 0; green, 0; blue, 0 }  ][fill={rgb, 255:red, 0; green, 0; blue, 0 }  ][line width=0.75]      (0, 0) circle [x radius= 3.35, y radius= 3.35]   ;
	%Straight Lines [id:da5323718096397894] 
	\draw    (50,150) -- (70,200) ;
	\draw [shift={(70,200)}, rotate = 68.2] [color={rgb, 255:red, 0; green, 0; blue, 0 }  ][fill={rgb, 255:red, 0; green, 0; blue, 0 }  ][line width=0.75]      (0, 0) circle [x radius= 3.35, y radius= 3.35]   ;
	%Straight Lines [id:da00953278801202151] 
	\draw    (150,150) -- (130,200) ;
	\draw [shift={(130,200)}, rotate = 111.8] [color={rgb, 255:red, 0; green, 0; blue, 0 }  ][fill={rgb, 255:red, 0; green, 0; blue, 0 }  ][line width=0.75]      (0, 0) circle [x radius= 3.35, y radius= 3.35]   ;
	%Straight Lines [id:da28820103879137715] 
	\draw    (150,150) -- (170,200) ;
	\draw [shift={(170,200)}, rotate = 68.2] [color={rgb, 255:red, 0; green, 0; blue, 0 }  ][fill={rgb, 255:red, 0; green, 0; blue, 0 }  ][line width=0.75]      (0, 0) circle [x radius= 3.35, y radius= 3.35]   ;
	%Straight Lines [id:da32654310531279407] 
	\draw    (250,150) -- (250,200) ;
	\draw [shift={(250,200)}, rotate = 90] [color={rgb, 255:red, 0; green, 0; blue, 0 }  ][fill={rgb, 255:red, 0; green, 0; blue, 0 }  ][line width=0.75]      (0, 0) circle [x radius= 3.35, y radius= 3.35]   ;
	%Straight Lines [id:da9070920526942725] 
	\draw    (350,150) -- (330,200) ;
	\draw [shift={(330,200)}, rotate = 111.8] [color={rgb, 255:red, 0; green, 0; blue, 0 }  ][fill={rgb, 255:red, 0; green, 0; blue, 0 }  ][line width=0.75]      (0, 0) circle [x radius= 3.35, y radius= 3.35]   ;
	%Straight Lines [id:da28971830860405956] 
	\draw    (350,150) -- (370,200) ;
	\draw [shift={(370,200)}, rotate = 68.2] [color={rgb, 255:red, 0; green, 0; blue, 0 }  ][fill={rgb, 255:red, 0; green, 0; blue, 0 }  ][line width=0.75]      (0, 0) circle [x radius= 3.35, y radius= 3.35]   ;
	%Straight Lines [id:da6600484670317939] 
	\draw    (450,150) -- (450,200) ;
	\draw [shift={(450,200)}, rotate = 90] [color={rgb, 255:red, 0; green, 0; blue, 0 }  ][fill={rgb, 255:red, 0; green, 0; blue, 0 }  ][line width=0.75]      (0, 0) circle [x radius= 3.35, y radius= 3.35]   ;
	
	\end{tikzpicture}
	\caption{The caterpillar $\C(3,4,3,4,2)$.}
	\label{fig:cat}
\end{figure}

Back to the proof of Theorem \ref{thm1}, since $G$ is unigraphic, again by the arguments in \cite{Maffucci 1} either $k=0$, or the caterpillar is a star centred at $a$, or both (for instance, the graphs $\C(3,4,3,4,2)\cup K_2$ and $\C(3,4,3,4)\cup S_2$ share the same sequence). Let us inspect these cases in turn.% Since $\deg_R(a)=\deg_G(a)\geq 4$, it follows that $A$ is a star with central vertex $a$, and $P_1=\dots=P_k=K_2$.

\begin{enumerate}
\item
Assume that $A$ is not a star, so that in particular $k=0$. Here $a$ is at one end of the central path in $A$ (otherwise there would be more than one vertex of degree $3$ in $G$). We can write
\[A=\C(m,j_2,\dots,j_l),\]
with
\[m:=\deg(a),\]
$l\geq 2$, and each of $j_2,\dots,j_l\geq 2$.

If $m\geq 5$, replacing $A$ with 
\[\C(4,j_2,\dots,j_l,m-3)\]
changes $G$ but not its sequence. We can rule out this case.

\begin{figure}[h!]
\centering
\includegraphics[width=6cm]{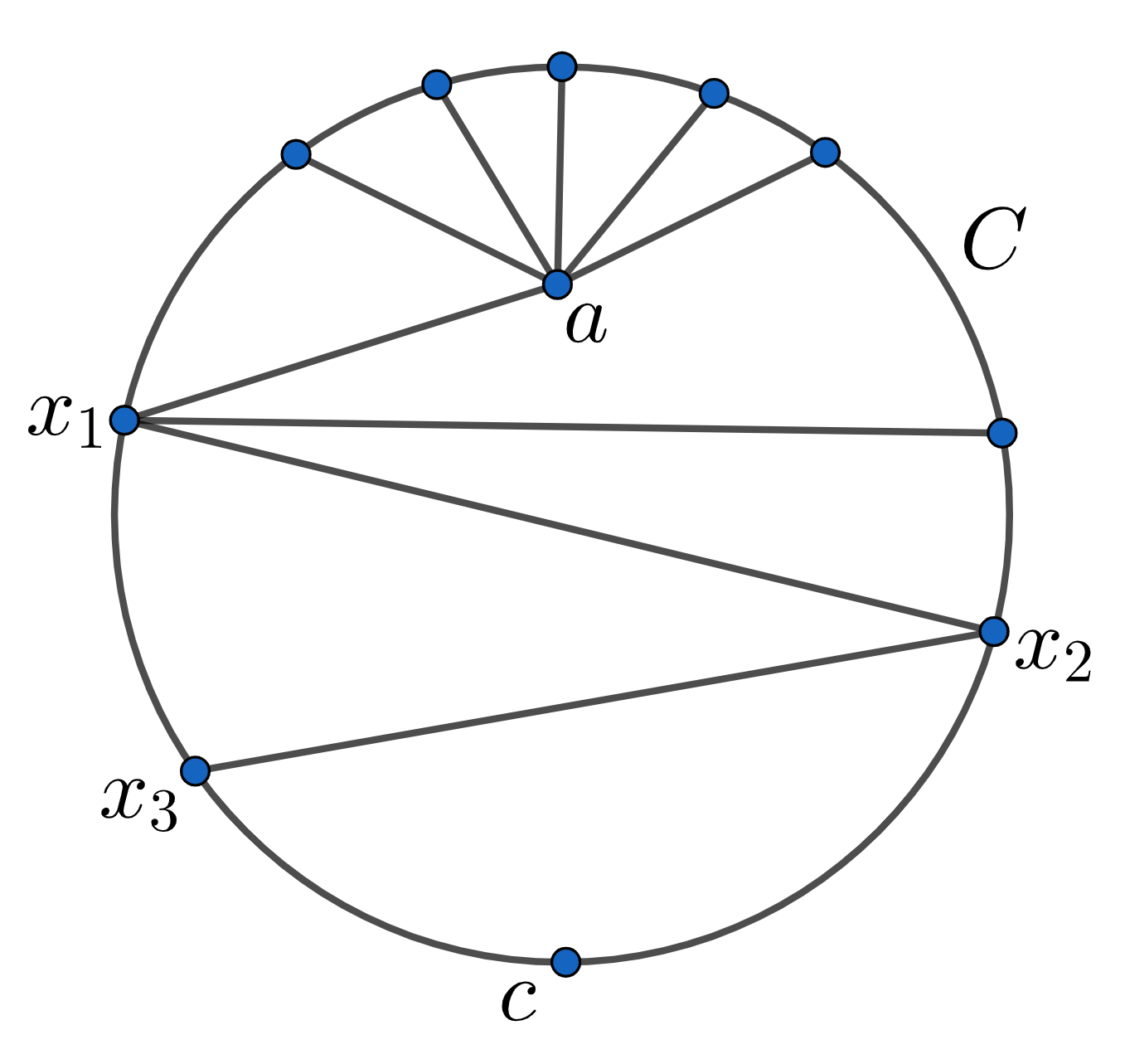}
\hspace{2cm}
\includegraphics[width=6cm]{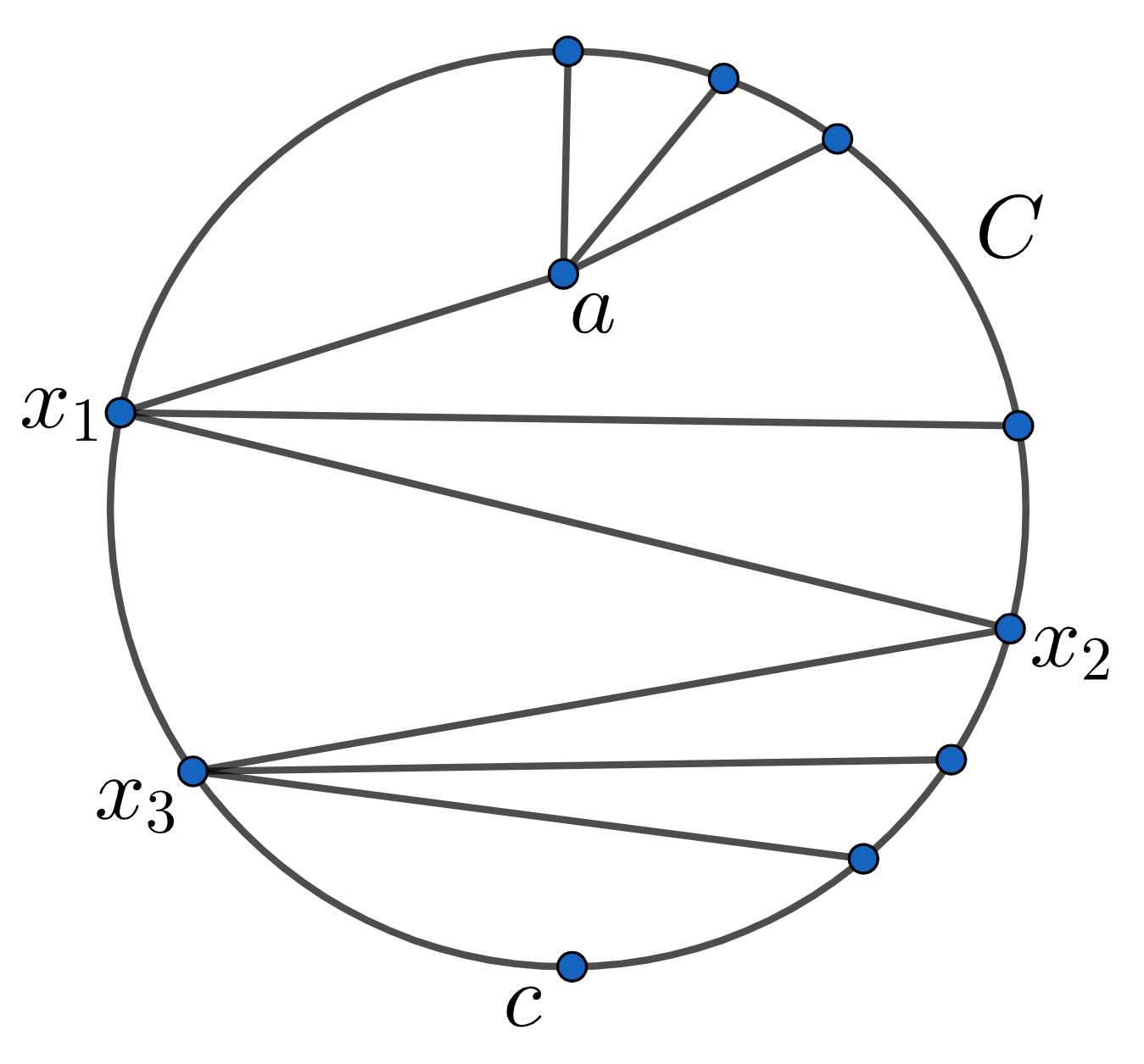}
\caption{Left: for $p=12$, the graph $F=G-y$ in the case of $R=A=\C(6,3,2)$ (so that $m=\deg(a)=6$). The sequence of $G$ is $10,6,6,5,4^7,3$. Right: the graph obtained by replacing $A$ with $\C(4,3,2,3)$. Note that the degrees of $a$ and $x_3$ have swapped.}
\label{fig:for1}
\end{figure}

If instead $m=4$, we look at $j_l$. If $j_l\geq 3$, we replace $A$ with
\[\C(5,j_2,\dots,j_{l-1},2).\]
If $j_l=2$, we replace $A$ with
\[\C(5,j_2,\dots,j_{l-1}).\]
In any case, we have changed $G$ but not its sequence.

\item
We are left with the case of $A$ being a star with central vertex $a$, and $P_1=\dots=P_k=K_2$ with $k\geq 0$. The degree sequence of $G$ is thus
\begin{equation}
\label{l1}
p-2,m,4^{(m+2k)},3.
\end{equation}
%where $m=\deg(a)$.

If $m\geq 5$, it is easy to see that, for any $k\geq 0$,
\[R=S_m\cup\underbrace{K_2\cup\dots\cup K_2}_{k}\]
and
\[R=\C(4,m-3)\cup\underbrace{K_2\cup\dots\cup K_2}_{k}\]
produce two non-isomorphic polyhedra, both of sequence \eqref{l1}.
%one can build two realisations of the sequence \eqref{l1}. In one, $R$ is the disjoint union
%\[R=S_{m}\cup\underbrace{K_2\cup\dots\cup K_2}_{k}\]
%(i.e., $a$ has degree $m$ in this polyhedron).
%In the other,
%\[R=S_{m-1}\cup S_{2}\cup\underbrace{K_2\cup\dots\cup K_2}_{k-1}.\]

%In one, $R$ consists of the disjoint union of $k$ copies of $K_2$ and one star graph $S_m$. The other possible realisation is with $A=\C(m-3, 4)$ and the same number of copies of $K_2$, as illustrated in Figure \ref{fig:forest}.%. Then the sequence obtained is
%\[\lambda=p-2,m,4,4^Q,3\] and $$Q=2k+q(A)=2k+ (m-5)+(4-2)+2=2k+m-1$$ and $\lambda$ coincides with $\lambda_1$. 

%\bigskip

\tikzset{every picture/.style={line width=0.75pt}} %set default line width to 0.75pt        

\begin{figure}[h!]
\centering

\begin{tikzpicture}[x=0.75pt,y=0.75pt,yscale=-0.8,xscale=0.8]
%uncomment if require: \path (0,381); %set diagram left start at 0, and has height of 381

%Shape: Ellipse [id:dp8445728136566177] 
\draw   (15,181.5) .. controls (15,116.05) and (67.34,63) .. (131.9,63) .. controls (196.46,63) and (248.8,116.05) .. (248.8,181.5) .. controls (248.8,246.95) and (196.46,300) .. (131.9,300) .. controls (67.34,300) and (15,246.95) .. (15,181.5) -- cycle ;
%Straight Lines [id:da17888806876748964] 
\draw    (66.97,83.31) -- (133.1,140.19) ;
\draw [shift={(133.1,140.19)}, rotate = 40.7] [color={rgb, 255:red, 0; green, 0; blue, 0 }  ][fill={rgb, 255:red, 0; green, 0; blue, 0 }  ][line width=0.75]      (0, 0) circle [x radius= 3.35, y radius= 3.35]   ;
\draw [shift={(66.97,83.31)}, rotate = 40.7] [color={rgb, 255:red, 0; green, 0; blue, 0 }  ][fill={rgb, 255:red, 0; green, 0; blue, 0 }  ][line width=0.75]      (0, 0) circle [x radius= 3.35, y radius= 3.35]   ;
%Straight Lines [id:da25538456620102146] 
\draw    (133.1,140.19) -- (99.03,69.77) ;
\draw [shift={(99.03,69.77)}, rotate = 244.18] [color={rgb, 255:red, 0; green, 0; blue, 0 }  ][fill={rgb, 255:red, 0; green, 0; blue, 0 }  ][line width=0.75]      (0, 0) circle [x radius= 3.35, y radius= 3.35]   ;
%Straight Lines [id:da6409286853672282] 
\draw    (133.1,140.19) -- (131.9,63) ;
\draw [shift={(131.9,63)}, rotate = 269.11] [color={rgb, 255:red, 0; green, 0; blue, 0 }  ][fill={rgb, 255:red, 0; green, 0; blue, 0 }  ][line width=0.75]      (0, 0) circle [x radius= 3.35, y radius= 3.35]   ;
%Straight Lines [id:da4489584615018125] 
\draw    (133.1,140.19) -- (167.17,68.42) ;
\draw [shift={(167.17,68.42)}, rotate = 295.39] [color={rgb, 255:red, 0; green, 0; blue, 0 }  ][fill={rgb, 255:red, 0; green, 0; blue, 0 }  ][line width=0.75]      (0, 0) circle [x radius= 3.35, y radius= 3.35]   ;
%Straight Lines [id:da8055448211329173] 
\draw    (133.1,140.19) -- (198.57,83.99) ;
\draw [shift={(198.57,83.99)}, rotate = 319.35] [color={rgb, 255:red, 0; green, 0; blue, 0 }  ][fill={rgb, 255:red, 0; green, 0; blue, 0 }  ][line width=0.75]      (0, 0) circle [x radius= 3.35, y radius= 3.35]   ;
%Curve Lines [id:da8415106995312904] 
\draw    (77.66,285.1) .. controls (99.03,266.82) and (165.83,266.14) .. (188.55,285.1) ;
\draw [shift={(188.55,285.1)}, rotate = 39.86] [color={rgb, 255:red, 0; green, 0; blue, 0 }  ][fill={rgb, 255:red, 0; green, 0; blue, 0 }  ][line width=0.75]      (0, 0) circle [x radius= 3.35, y radius= 3.35]   ;
\draw [shift={(77.66,285.1)}, rotate = 319.46] [color={rgb, 255:red, 0; green, 0; blue, 0 }  ][fill={rgb, 255:red, 0; green, 0; blue, 0 }  ][line width=0.75]      (0, 0) circle [x radius= 3.35, y radius= 3.35]   ;
%Curve Lines [id:da9176653577866536] 
\draw    (56.28,271.56) .. controls (101.57,243.8) and (160.49,246.51) .. (207.25,271.56) ;
\draw [shift={(207.25,271.56)}, rotate = 28.18] [color={rgb, 255:red, 0; green, 0; blue, 0 }  ][fill={rgb, 255:red, 0; green, 0; blue, 0 }  ][line width=0.75]      (0, 0) circle [x radius= 3.35, y radius= 3.35]   ;
\draw [shift={(56.28,271.56)}, rotate = 328.49] [color={rgb, 255:red, 0; green, 0; blue, 0 }  ][fill={rgb, 255:red, 0; green, 0; blue, 0 }  ][line width=0.75]      (0, 0) circle [x radius= 3.35, y radius= 3.35]   ;
%Straight Lines [id:da8567360911102075] 
\draw    (131.9,300) ;
\draw [shift={(131.9,300)}, rotate = 0] [color={rgb, 255:red, 0; green, 0; blue, 0 }  ][fill={rgb, 255:red, 0; green, 0; blue, 0 }  ][line width=0.75]      (0, 0) circle [x radius= 3.35, y radius= 3.35]   ;
%Shape: Ellipse [id:dp3885611175792014] 
\draw   (383,182.5) .. controls (383,117.05) and (435.34,64) .. (499.9,64) .. controls (564.46,64) and (616.8,117.05) .. (616.8,182.5) .. controls (616.8,247.95) and (564.46,301) .. (499.9,301) .. controls (435.34,301) and (383,247.95) .. (383,182.5) -- cycle ;
%Straight Lines [id:da8138470186207165] 
\draw    (501.1,141.19) -- (386.8,154) ;
\draw [shift={(386.8,154)}, rotate = 173.61] [color={rgb, 255:red, 0; green, 0; blue, 0 }  ][fill={rgb, 255:red, 0; green, 0; blue, 0 }  ][line width=0.75]      (0, 0) circle [x radius= 3.35, y radius= 3.35]   ;
\draw [shift={(501.1,141.19)}, rotate = 173.61] [color={rgb, 255:red, 0; green, 0; blue, 0 }  ][fill={rgb, 255:red, 0; green, 0; blue, 0 }  ][line width=0.75]      (0, 0) circle [x radius= 3.35, y radius= 3.35]   ;
%Straight Lines [id:da738149888740333] 
\draw    (501.1,141.19) -- (467.03,70.77) ;
\draw [shift={(467.03,70.77)}, rotate = 244.18] [color={rgb, 255:red, 0; green, 0; blue, 0 }  ][fill={rgb, 255:red, 0; green, 0; blue, 0 }  ][line width=0.75]      (0, 0) circle [x radius= 3.35, y radius= 3.35]   ;
%Straight Lines [id:da09586804553302164] 
\draw    (501.1,141.19) -- (499.9,64) ;
\draw [shift={(499.9,64)}, rotate = 269.11] [color={rgb, 255:red, 0; green, 0; blue, 0 }  ][fill={rgb, 255:red, 0; green, 0; blue, 0 }  ][line width=0.75]      (0, 0) circle [x radius= 3.35, y radius= 3.35]   ;
%Straight Lines [id:da7123383368129261] 
\draw    (501.1,141.19) -- (535.17,69.42) ;
\draw [shift={(535.17,69.42)}, rotate = 295.39] [color={rgb, 255:red, 0; green, 0; blue, 0 }  ][fill={rgb, 255:red, 0; green, 0; blue, 0 }  ][line width=0.75]      (0, 0) circle [x radius= 3.35, y radius= 3.35]   ;
%Straight Lines [id:da004663549307159176] 
\draw    (386.8,154) -- (616.8,193) ;
\draw [shift={(616.8,193)}, rotate = 9.62] [color={rgb, 255:red, 0; green, 0; blue, 0 }  ][fill={rgb, 255:red, 0; green, 0; blue, 0 }  ][line width=0.75]      (0, 0) circle [x radius= 3.35, y radius= 3.35]   ;
%Curve Lines [id:da0011741775793412401] 
\draw    (445.66,286.1) .. controls (467.03,267.82) and (533.83,267.14) .. (556.55,286.1) ;
\draw [shift={(556.55,286.1)}, rotate = 39.86] [color={rgb, 255:red, 0; green, 0; blue, 0 }  ][fill={rgb, 255:red, 0; green, 0; blue, 0 }  ][line width=0.75]      (0, 0) circle [x radius= 3.35, y radius= 3.35]   ;
\draw [shift={(445.66,286.1)}, rotate = 319.46] [color={rgb, 255:red, 0; green, 0; blue, 0 }  ][fill={rgb, 255:red, 0; green, 0; blue, 0 }  ][line width=0.75]      (0, 0) circle [x radius= 3.35, y radius= 3.35]   ;
%Curve Lines [id:da17955197977844373] 
\draw    (424.28,272.56) .. controls (469.57,244.8) and (528.49,247.51) .. (575.25,272.56) ;
\draw [shift={(575.25,272.56)}, rotate = 28.18] [color={rgb, 255:red, 0; green, 0; blue, 0 }  ][fill={rgb, 255:red, 0; green, 0; blue, 0 }  ][line width=0.75]      (0, 0) circle [x radius= 3.35, y radius= 3.35]   ;
\draw [shift={(424.28,272.56)}, rotate = 328.49] [color={rgb, 255:red, 0; green, 0; blue, 0 }  ][fill={rgb, 255:red, 0; green, 0; blue, 0 }  ][line width=0.75]      (0, 0) circle [x radius= 3.35, y radius= 3.35]   ;
%Straight Lines [id:da7215640046698637] 
\draw    (499.9,301) ;
\draw [shift={(499.9,301)}, rotate = 0] [color={rgb, 255:red, 0; green, 0; blue, 0 }  ][fill={rgb, 255:red, 0; green, 0; blue, 0 }  ][line width=0.75]      (0, 0) circle [x radius= 3.35, y radius= 3.35]   ;

% Text Node
\draw (57,24) node [anchor=north west][inner sep=0.75pt]   [align=left] {$\displaystyle k=2,\ A=S_{5}$};
% Text Node
\draw (433,27) node [anchor=north west][inner sep=0.75pt]   [align=left] {$\displaystyle k=2,\ A=\mathcal{C}( 2,4)$};
% Text Node
\draw (230,318) node [anchor=north west][inner sep=0.75pt]  [font=\large] [align=left] {$\displaystyle \lambda_1^{12}=p-2,5,4^{9} ,3$};
% Text Node
\draw (290,152) node [anchor=north west][inner sep=0.75pt]  [font=\huge] [align=left] {$\displaystyle \ncong $};

\end{tikzpicture}
\caption{Two realisations of the sequence \eqref{l1} when $m=5$ and $k=2$ (sketches of $F$).}
\label{fig:forest}
\end{figure}
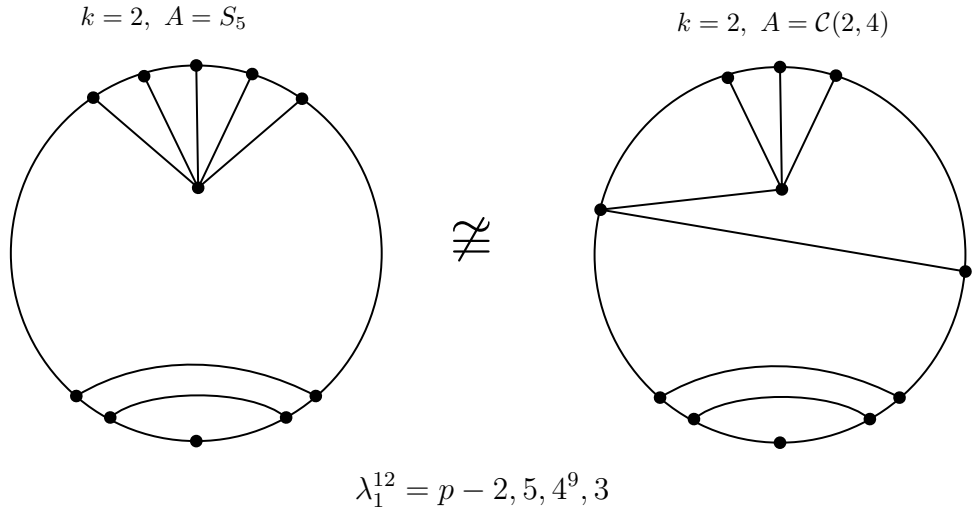
%\bigskip

It follows that the only candidate unigraphic sequence for $G$ when $R$ is a forest is \eqref{l1} in the case $m=4$, i.e.
\[p-2,4^{(2k+5)},3.\] Since $2k+5+2=p$, we have $k=(p-7)/2$ and $p$ is necessarily odd. On the other hand, we have the following.

\begin{lemma}
    The sequence $\gamma(p)=p-2,4^{p-2},3$ with odd $p\geq 7$ is unigraphic.
\end{lemma}
\begin{proof}
We already know that $\deg(a)\geq 4$, thus here $\deg(a)=4$. Therefore, the sequence of $R=G-y-E(C)-c$ is
\[4,1^{p-3}.\]
The only graph of such sequence is indeed the disjoint union of $S_4$ with $a$ at the centre, and $(p-7)/2$ copies of $K_2$.
\end{proof}
\end{enumerate}

\subsection{Second case: $R$ contains a triangle}
Let $v,w$ be the other two vertices of this triangle apart than $a$. By prior considerations (cf. Lemma \ref{max 1 triangle}) if we consider the cycle $C$, the vertices $a,c$ appear on opposite sides of the edge $vw$. Moreover, each vertex of $R=F-E(C)-c$ adjacent to $a$ other than $v,w$ is of degree exactly $4$ in $G$, otherwise there would be other degree $3$ vertices in $G$ apart from $c$. For the same reason, at least one of $v,w$ has degree $5$, say
\[\deg_G(w)=5.\]
Again by the same argument, all vertices adjacent to $v$ in $R$ other than $a$ and $w$ (if any others exist) lie on $C$ on the $cw$-path not containing $v$, and at most one of these vertices (the one closest to $c$ along this $cw$-path) can be of degree $5$ or higher. We summarise these consideration in the following result.

\begin{lemma}
	\label{lem:R}
	Let $S$ be the graph $R-w-av$. Then $S$ is the disjoint union
\[S=S_{m-2}\cup T\cup P_1\cup\dots\cup P_k,\]	
	where $S_{m-2}$ is a star centred at $a$, $T$ is a caterpillar containing $v$, and $P_1,\dots,P_k$ are $k\geq 0$ copies of $K_2$.
\end{lemma}
If we understand $m=\deg(a)$, $T$, and $k$, we will be able to recover $R$ and thus $G$. For instance, if we know that $m=7$, $T=\C(2,3)$ with $v$ the non-leaf vertex of $T$ of degree $j_1=2$ in $T$, and $k=1$, we construct a graph as illustrated in Figure \ref{fig:S1}.

\begin{figure}[h!]
	\centering
	\includegraphics[width=6cm]{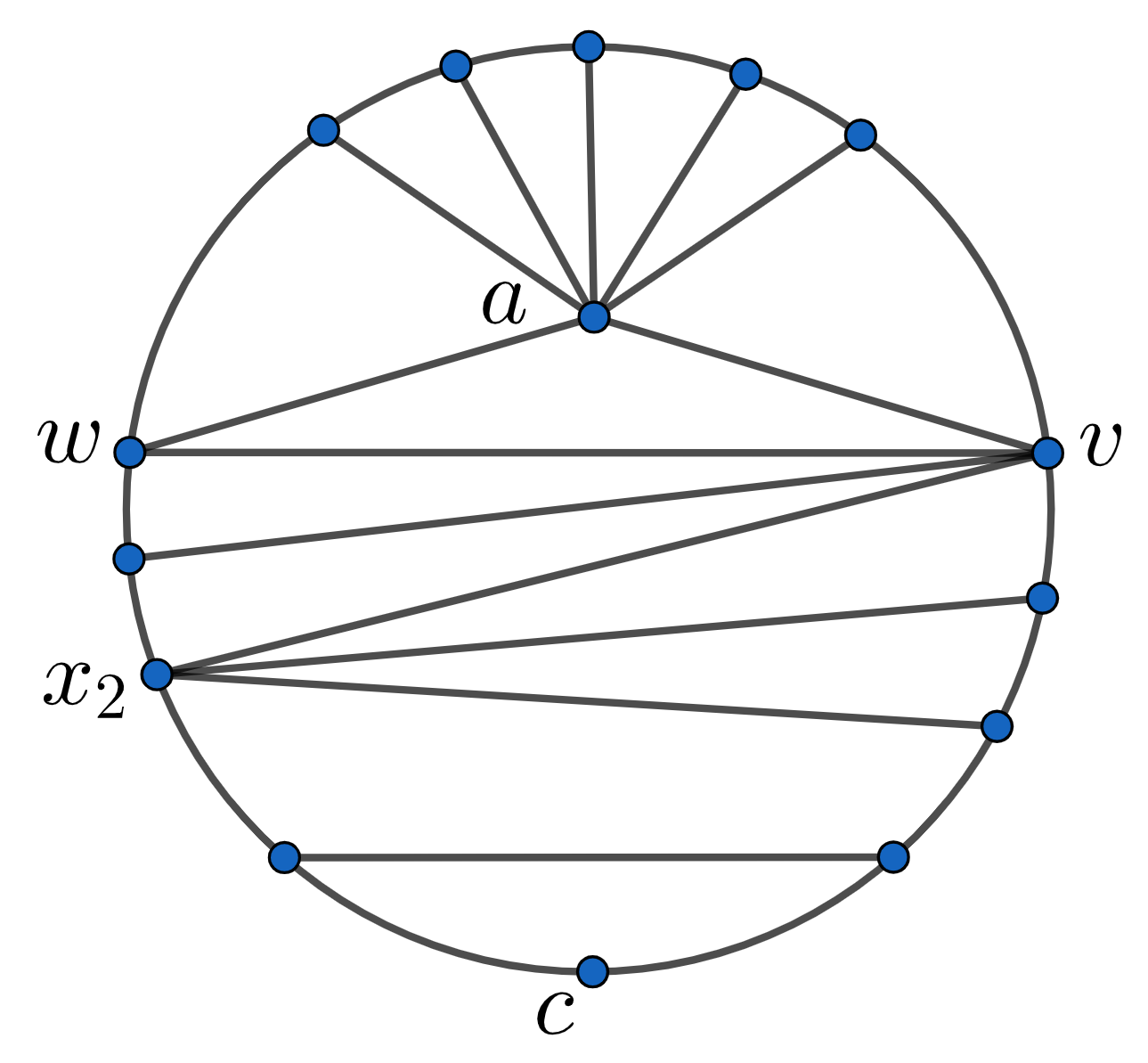}
	\caption{In this example, given $m=7$, $T=\C(2,3)$, with central path $v,x_2$, and $k=1$, we sketch $F=G-y$ as shown, with $p=1+m+j_1+(j_2-1)+2k+1+1=16$. Recall that $y$ is adjacent to all vertices except $a$.}
	\label{fig:S1}
\end{figure}

We begin our analysis of $m,T,k$ with the following auxiliary result.

\begin{lemma}
	\label{lem:4}
	Let $G$ be unigraphic and $R$ contain a triangle. Suppose either that $T$ is non-trivial, or that $k\geq 1$ (or both). Then $m=4$.
\end{lemma}
\begin{proof}
	By the hypotheses, there is in $T$ or in one of the $P_i$ a pendant vertex $x\neq v$, that is adjacent to the vertex $c$ in the cycle $C$ of $F$. By contradiction, let $m\geq 5$. We delete $m-4$ vertices
	\[u_1,\dots,u_{m-4}\]
	of degree $4$ in $G$ that are adjacent to $a$, and then insert vertices
	\[z_1,\dots,z_{m-4}\]
	and edges
	\[xz_1,\dots,xz_{m-4}\]
	in $C$, such that
	\[x,c,z_1,\dots,z_{m-4}\]
	is a path in $C$. The degrees of $a$ and $x$ have swapped, hence the polyhedron $G$ is not unigraphic. This construction is very similar to the one in Figure \ref{fig:for1}.
\end{proof}

We now inspect in turn the possibilities of $T$ being the trivial graph or not.

\subsubsection{$T$ is trivial}
\label{sec:triv}
%In this case, $R-v-w$ is the disjoint union
%\begin{equation}
%\label{eqn:un}
%S_{m-2}\cup\underbrace{K_2\cup\dots\cup K_2}_{k}.
%\end{equation}
%However, if $m\geq 5$ and $k\geq 1$, we could just replace \eqref{eqn:un} with
%\begin{equation*}
%S_{m-3}\cup S_{2}\cup\underbrace{K_2\cup\dots\cup K_2}_{k-1},
%\end{equation*}
%changing $G$ but not its degree sequence ($a$ would have degree $m-1$ in the new polyhedron).
Here the sequence of $G$ is
\begin{equation}
\label{l2}
p-2,m,5,5,4^{m-2+2k},3.
\end{equation}
We analyse two subcases.
\begin{enumerate}
	%\item
	%If $m\geq 6$, to see that \eqref{l2} is not unigraphic for any $k\geq 0$, we replace $S$ (union of $S_{m-2}$, $T=K_1$, and $k$ copies of $K_2$) with the union of $S_{3}$, $T=S_{m-5}$, and the same number of copies of $K_2$. An illustration may be found in Figure \ref{fig:tr1}.
	\item
	If $k\geq 1$, then $m=4$ by Lemma \ref{lem:4}. 
	To see that this sequence is not unigraphic, we replace $R=G-E(C)-c$ with the union of $\C(5,2)$ and $k-1$ copies of $K_2$ (in particular, in the new graph, $R$ is a forest and $a$ is the non-leaf vertex of $\C(5,2)$ with degree $5$).
	\item
	If $k=0$, then \eqref{l2} reads $p-2,m,5,5,4^{m-2},3$. We have $m-2=p-5$, hence this sequence is simply
	\begin{equation}
	\label{eqn:nu}
	p-2,p-3,5,5,4^{p-5},3
	\end{equation}
	with $p\geq 7$. For the values $p=7,8$ we find $\alpha(7)=5^3,4^3,3$ and the exceptional $6,5^3,4^3,3$. Now let $p\geq 9$, i.e. $m\geq 6$. In the graph $S$, we are left with a star $S_{m-2}$ centred at $a$, and a copy of $K_1$ given by $v$. Similarly to before, replacing the copy of $S_{m-2}$ with $S_{3}$ and $K_1$ with $S_{m-5}$ alters $G$ but not \eqref{eqn:nu}, so that \eqref{eqn:nu} is not unigraphic for any $p\geq 9$.
\end{enumerate}

\subsubsection{$T$ is non-trivial}
By Lemma \ref{lem:4}, we already know that $m=4$. We need to find the caterpillar $T$ and the number of copies $k$ of $K_2$. In the following two lemmas, assuming that $G$ is unigraphic, we will show that $T$ is a star centred at $v$, and that $k=0$. We will use similar ideas to the case of $R$ being a forest.

\begin{lemma}
	\label{lem:star}
If $G$ is unigraphic, $R$ contains a triangle, and the caterpillar $T$ is non-trivial, then $T$ is a star centred at $v$.
\end{lemma}
\begin{proof}
Arguing by contradiction, we write
\[T=\C(j_1,\dots,j_l),\]
where $l\geq 1$ and $v$ is not the vertex represented by the degree $j_l$ in $T$. By Lemma \ref{lem:4}, we know that $m=4$. As in section \ref{sec:for}, case 1., %if $m\geq 5$ we replace the copy of $S_{m-2}$ with $S_{2}$ in the graph $S$, and $T$ with
%\[T=\C(j_1,\dots,j_l,m-3)\]
%(i.e., $a$ has degree $4$ in this new polyhedron).
%If instead $m=4$,
we replace $S_{2}$ with $S_{3}$ in the graph $S$, and $T$ with
\[T=\C(j_1,\dots,j_l-1)\]
or
\[\C(j_1,\dots,j_{l-1})\]
depending on whether $j_l\geq 3$ or $j_l=2$ ($a$ has degree $5$ in this new polyhedron). Therefore, $G$ is not unigraphic.
\end{proof}

\begin{lemma}
	\label{lem:k}
If $G$ is unigraphic, $R$ contains a triangle, and the caterpillar $T$ is non-trivial, then $k=0$.
\end{lemma}
\begin{proof}
We use Lemmas \ref{lem:4} and \ref{lem:star}. If $k\geq 1$, it is easy to modify $G$ but not its sequence, as in Figure \ref{fig:tr2}.

	%The situation with $m=4, k=2, j=7$ :
	%\bigskip
	
	\tikzset{every picture/.style={line width=0.75pt}} %set default line width to 0.75pt        
	\begin{figure}[h!]
		\centering
		\begin{tikzpicture}[x=0.75pt,y=0.75pt,yscale=-0.8,xscale=0.8]
		%uncomment if require: \path (0,381); %set diagram left start at 0, and has height of 381
		
		%Shape: Ellipse [id:dp8445728136566177] 
		\draw   (392,167.5) .. controls (392,102.05) and (444.34,49) .. (508.9,49) .. controls (573.46,49) and (625.8,102.05) .. (625.8,167.5) .. controls (625.8,232.95) and (573.46,286) .. (508.9,286) .. controls (444.34,286) and (392,232.95) .. (392,167.5) -- cycle ;
		%Straight Lines [id:da25538456620102146] 
		\draw    (510.1,126.19) -- (476.03,55.77) ;
		\draw [shift={(476.03,55.77)}, rotate = 244.18] [color={rgb, 255:red, 0; green, 0; blue, 0 }  ][fill={rgb, 255:red, 0; green, 0; blue, 0 }  ][line width=0.75]      (0, 0) circle [x radius= 3.35, y radius= 3.35]   ;
		%Straight Lines [id:da4489584615018125] 
		\draw    (510.1,126.19) -- (544.17,54.42) ;
		\draw [shift={(544.17,54.42)}, rotate = 295.39] [color={rgb, 255:red, 0; green, 0; blue, 0 }  ][fill={rgb, 255:red, 0; green, 0; blue, 0 }  ][line width=0.75]      (0, 0) circle [x radius= 3.35, y radius= 3.35]   ;
		%Straight Lines [id:da8567360911102075] 
		\draw    (508.9,286) ;
		\draw [shift={(508.9,286)}, rotate = 0] [color={rgb, 255:red, 0; green, 0; blue, 0 }  ][fill={rgb, 255:red, 0; green, 0; blue, 0 }  ][line width=0.75]      (0, 0) circle [x radius= 3.35, y radius= 3.35]   ;
		%Straight Lines [id:da9961564606090281] 
		\draw    (510.1,126.19) -- (625.8,174) ;
		\draw [shift={(510.1,126.19)}, rotate = 22.45] [color={rgb, 255:red, 0; green, 0; blue, 0 }  ][fill={rgb, 255:red, 0; green, 0; blue, 0 }  ][line width=0.75]      (0, 0) circle [x radius= 3.35, y radius= 3.35]   ;
		%Straight Lines [id:da07909734647426614] 
		\draw    (510.1,126.19) -- (394.8,141) ;
		\draw [shift={(394.8,141)}, rotate = 172.68] [color={rgb, 255:red, 0; green, 0; blue, 0 }  ][fill={rgb, 255:red, 0; green, 0; blue, 0 }  ][line width=0.75]      (0, 0) circle [x radius= 3.35, y radius= 3.35]   ;
		%Straight Lines [id:da8457675653987344] 
		\draw    (392.8,183) -- (625.8,174) ;
		\draw [shift={(625.8,174)}, rotate = 357.79] [color={rgb, 255:red, 0; green, 0; blue, 0 }  ][fill={rgb, 255:red, 0; green, 0; blue, 0 }  ][line width=0.75]      (0, 0) circle [x radius= 3.35, y radius= 3.35]   ;
		\draw [shift={(392.8,183)}, rotate = 357.79] [color={rgb, 255:red, 0; green, 0; blue, 0 }  ][fill={rgb, 255:red, 0; green, 0; blue, 0 }  ][line width=0.75]      (0, 0) circle [x radius= 3.35, y radius= 3.35]   ;
		%Curve Lines [id:da7475960501355896] 
		\draw    (425.8,249) .. controls (485.8,235) and (605.8,212) .. (625.8,174) ;
		\draw [shift={(625.8,174)}, rotate = 297.76] [color={rgb, 255:red, 0; green, 0; blue, 0 }  ][fill={rgb, 255:red, 0; green, 0; blue, 0 }  ][line width=0.75]      (0, 0) circle [x radius= 3.35, y radius= 3.35]   ;
		\draw [shift={(425.8,249)}, rotate = 346.87] [color={rgb, 255:red, 0; green, 0; blue, 0 }  ][fill={rgb, 255:red, 0; green, 0; blue, 0 }  ][line width=0.75]      (0, 0) circle [x radius= 3.35, y radius= 3.35]   ;
		%Curve Lines [id:da3912644277331079] 
		\draw    (408.8,227) .. controls (448.8,197) and (565.8,196) .. (625.8,174) ;
		\draw [shift={(408.8,227)}, rotate = 323.13] [color={rgb, 255:red, 0; green, 0; blue, 0 }  ][fill={rgb, 255:red, 0; green, 0; blue, 0 }  ][line width=0.75]      (0, 0) circle [x radius= 3.35, y radius= 3.35]   ;
		%Curve Lines [id:da044438064618592676] 
		\draw    (477,281) .. controls (492.8,265) and (522.8,268) .. (548.8,279) ;
		\draw [shift={(548.8,279)}, rotate = 22.93] [color={rgb, 255:red, 0; green, 0; blue, 0 }  ][fill={rgb, 255:red, 0; green, 0; blue, 0 }  ][line width=0.75]      (0, 0) circle [x radius= 3.35, y radius= 3.35]   ;
		\draw [shift={(477,281)}, rotate = 314.64] [color={rgb, 255:red, 0; green, 0; blue, 0 }  ][fill={rgb, 255:red, 0; green, 0; blue, 0 }  ][line width=0.75]      (0, 0) circle [x radius= 3.35, y radius= 3.35]   ;
		%Shape: Ellipse [id:dp15749577373242007] 
		\draw   (53,165.5) .. controls (53,100.05) and (105.34,47) .. (169.9,47) .. controls (234.46,47) and (286.8,100.05) .. (286.8,165.5) .. controls (286.8,230.95) and (234.46,284) .. (169.9,284) .. controls (105.34,284) and (53,230.95) .. (53,165.5) -- cycle ;
		%Straight Lines [id:da9429348357640484] 
		\draw    (171.1,124.19) -- (137.03,53.77) ;
		\draw [shift={(137.03,53.77)}, rotate = 244.18] [color={rgb, 255:red, 0; green, 0; blue, 0 }  ][fill={rgb, 255:red, 0; green, 0; blue, 0 }  ][line width=0.75]      (0, 0) circle [x radius= 3.35, y radius= 3.35]   ;
		%Straight Lines [id:da971749559949679] 
		\draw    (171.1,124.19) -- (205.17,52.42) ;
		\draw [shift={(205.17,52.42)}, rotate = 295.39] [color={rgb, 255:red, 0; green, 0; blue, 0 }  ][fill={rgb, 255:red, 0; green, 0; blue, 0 }  ][line width=0.75]      (0, 0) circle [x radius= 3.35, y radius= 3.35]   ;
		%Straight Lines [id:da3634777463382577] 
		\draw    (169.9,284) ;
		\draw [shift={(169.9,284)}, rotate = 0] [color={rgb, 255:red, 0; green, 0; blue, 0 }  ][fill={rgb, 255:red, 0; green, 0; blue, 0 }  ][line width=0.75]      (0, 0) circle [x radius= 3.35, y radius= 3.35]   ;
		%Straight Lines [id:da6272478128584429] 
		\draw    (171.1,124.19) -- (286.8,172) ;
		\draw [shift={(171.1,124.19)}, rotate = 22.45] [color={rgb, 255:red, 0; green, 0; blue, 0 }  ][fill={rgb, 255:red, 0; green, 0; blue, 0 }  ][line width=0.75]      (0, 0) circle [x radius= 3.35, y radius= 3.35]   ;
		%Straight Lines [id:da8974017017241098] 
		\draw    (171.1,124.19) -- (53.8,181) ;
		\draw [shift={(53.8,181)}, rotate = 154.16] [color={rgb, 255:red, 0; green, 0; blue, 0 }  ][fill={rgb, 255:red, 0; green, 0; blue, 0 }  ][line width=0.75]      (0, 0) circle [x radius= 3.35, y radius= 3.35]   ;
		%Straight Lines [id:da35407139339622984] 
		\draw    (53.8,181) -- (286.8,172) ;
		\draw [shift={(286.8,172)}, rotate = 357.79] [color={rgb, 255:red, 0; green, 0; blue, 0 }  ][fill={rgb, 255:red, 0; green, 0; blue, 0 }  ][line width=0.75]      (0, 0) circle [x radius= 3.35, y radius= 3.35]   ;
		%Curve Lines [id:da07769689919152567] 
		\draw    (69.8,225) .. controls (109.8,195) and (226.8,194) .. (286.8,172) ;
		\draw [shift={(69.8,225)}, rotate = 323.13] [color={rgb, 255:red, 0; green, 0; blue, 0 }  ][fill={rgb, 255:red, 0; green, 0; blue, 0 }  ][line width=0.75]      (0, 0) circle [x radius= 3.35, y radius= 3.35]   ;
		%Curve Lines [id:da8929734927055886] 
		\draw    (138,279) .. controls (153.8,263) and (183.8,266) .. (209.8,277) ;
		\draw [shift={(209.8,277)}, rotate = 22.93] [color={rgb, 255:red, 0; green, 0; blue, 0 }  ][fill={rgb, 255:red, 0; green, 0; blue, 0 }  ][line width=0.75]      (0, 0) circle [x radius= 3.35, y radius= 3.35]   ;
		\draw [shift={(138,279)}, rotate = 314.64] [color={rgb, 255:red, 0; green, 0; blue, 0 }  ][fill={rgb, 255:red, 0; green, 0; blue, 0 }  ][line width=0.75]      (0, 0) circle [x radius= 3.35, y radius= 3.35]   ;
		%Curve Lines [id:da5576612180112401] 
		\draw    (115.8,270) .. controls (143.8,243) and (204.8,249) .. (225.8,270) ;
		\draw [shift={(225.8,270)}, rotate = 45] [color={rgb, 255:red, 0; green, 0; blue, 0 }  ][fill={rgb, 255:red, 0; green, 0; blue, 0 }  ][line width=0.75]      (0, 0) circle [x radius= 3.35, y radius= 3.35]   ;
		\draw [shift={(115.8,270)}, rotate = 316.04] [color={rgb, 255:red, 0; green, 0; blue, 0 }  ][fill={rgb, 255:red, 0; green, 0; blue, 0 }  ][line width=0.75]      (0, 0) circle [x radius= 3.35, y radius= 3.35]   ;
		%Curve Lines [id:da3492324119033945] 
		\draw    (81.8,244) .. controls (128.8,242) and (254.8,208) .. (286.8,172) ;
		\draw [shift={(81.8,244)}, rotate = 357.56] [color={rgb, 255:red, 0; green, 0; blue, 0 }  ][fill={rgb, 255:red, 0; green, 0; blue, 0 }  ][line width=0.75]      (0, 0) circle [x radius= 3.35, y radius= 3.35]   ;
		%Straight Lines [id:da4510661504156528] 
		\draw    (508.9,49) -- (510.1,126.19) ;
		\draw [shift={(510.1,126.19)}, rotate = 89.11] [color={rgb, 255:red, 0; green, 0; blue, 0 }  ][fill={rgb, 255:red, 0; green, 0; blue, 0 }  ][line width=0.75]      (0, 0) circle [x radius= 3.35, y radius= 3.35]   ;
		\draw [shift={(508.9,49)}, rotate = 89.11] [color={rgb, 255:red, 0; green, 0; blue, 0 }  ][fill={rgb, 255:red, 0; green, 0; blue, 0 }  ][line width=0.75]      (0, 0) circle [x radius= 3.35, y radius= 3.35]   ;
		
		% Text Node
		\draw (448,318) node [anchor=north west][inner sep=0.75pt]   [align=left] {$\displaystyle \lambda'=11,7,5,4^{9} ,3$};
		% Text Node
		\draw (327,156) node [anchor=north west][inner sep=0.75pt]  [font=\LARGE] [align=left] {$\displaystyle \ncong $};
		% Text Node
		\draw (99,307) node [anchor=north west][inner sep=0.75pt]   [align=left] {$\displaystyle \lambda_{2}^{13} =11,7,5,4^{9} ,3$};

		\end{tikzpicture}
\caption{For instance, if $T$ is the star $S_2$ centred at $v$, and $k=2$ (left), we construct another polyhedron of same sequence, where $k=2-1=1$ (right).}
\label{fig:tr2}
\end{figure}
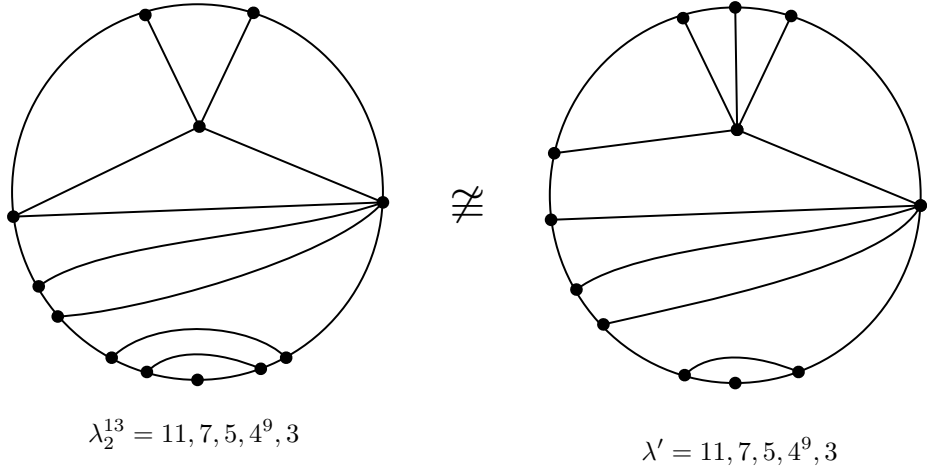
\end{proof}

Thanks to Lemmas \ref{lem:4}, \ref{lem:star}, and \ref{lem:k}, we are able to write the degree sequence of $G$ as
\[p-2,j,5,4^{(j-5)+1+2},3\]
with $j:=\deg_G(v)\geq 6$. It follows that $j-2+4=p$, so that the sequence of $G$ is simply
\[p-2,p-2,5,4^{p-4},3=\alpha(p),\]
$p\geq 8$. By Theorem \ref{thm2}, $\alpha$ is indeed unigraphic. The proof of Theorem \ref{thm1} is complete.

\section{Proof of Theorem \ref{thm3}}
Now let $p\geq 11$, and $\nu(p)$ be a non-increasing unigraphic polyhedral sequence beginning with
\[p-2,d_1,d_2,d_3,\dots,\]
where
\[d_1,d_2,d_3\geq 7\]
and
\begin{equation}
\label{ten}
d_1+d_2+d_3\geq p+10.
\end{equation}
We call $G$ a polyhedron of sequence $\nu$ on $p$ vertices, $y$ the vertex of degree $p-2$, and $a$ the only vertex not adjacent to $y$. According to Lemma \ref{bigcycle}, there is a cycle $C$ containing all vertices of $G-y-a$. %Moreover, let $x_1,x_2$ be the two vertices of degree $m+6$, and $x_3$ the vertex of degree $p-2m-2$ in $G$.
%Note that for any $m$, if $G$ is polyhedral, then it is maximal planar.

We write
\[V(G)=\{y\}\cup\{x_1,x_2,x_3\}\cup U,\]
where $\deg(x_i)=d_i$ for $i=1,2,3$. % We have
%\begin{equation}
%\label{ten}
%\deg(x_1)+\deg(x_2)+\deg(x_3)=p+10
%\end{equation}
%by assumption.
Independently of the choice of $a$, planarity does not allow for three elements of $U$ to be all adjacent to $x_1$ and $x_2$, say. 

\begin{lemma}
	\label{lem:a}
	We have $a\in U$, and
	\[ax_1,ax_2,ax_3\in E(G).\]
\end{lemma}
\begin{proof}
	The strategy is to show that in all other cases, we get $\deg(x_1)+\deg(x_2)+\deg(x_3)<p+10$, contradicting \eqref{ten}. Firstly, we distinguish between whether or not $a\in\{x_1,x_2,x_3\}$. 
	\begin{enumerate}
		\item
		$a\in\{x_1,x_2,x_3\}$.  %Then exactly two of $x_1,x_2,x_3$ are adjacent to $y$.
		There are a few subcases.
		\begin{enumerate}
			\item
			Say that two elements of $U$ are both adjacent to all three of $x_1,x_2,x_3$. Then by planarity, all $p-6$ other elements of $U$ are adjacent each to at most one of $x_1,x_2,x_3$. 
			%, and moreover there are at most two edges between $x_1,x_2,x_3$.
			%Even if $x_1,x_2,x_3$ are all pairwise adjacent.
			It follows that
			\begin{align*}
			d_1+d_2+d_3&\leq |\{yx_1,yx_2,yx_3\}|+2|\{x_1x_2,x_1x_3,x_2x_3\}|+2\cdot 3+(p-6)
			\\&=3+6+6+(p-6)=p+9.
			\end{align*}
			\item
			Say that exactly one element $z_1$ of $U$ is adjacent to all of $x_1,x_2,x_3$. By planarity, at most two elements of $U\setminus\{z_1\}$ may be adjacent to two of $x_1,x_2,x_3$. In this case we have
			\[d_1+d_2+d_3\leq 3+6+|\{z_1x_1,z_1x_2,z_1x_3\}|+2+2+(p-4-1-2)=p+9.\]
			\item
			The remaining case is, no element of $U$ is adjacent to all three of $x_1,x_2,x_3$. %Even if
			%\[z_1x_1,z_1x_2,z_2x_1,z_2x_2,z_3x_1,z_3x_3,z_4x_1,z_4x_3,z_5x_2,z_5x_3,z_6x_2,z_6x_3\]
			%are all edges of $G$, with $z_1,\dots,z_6$ distinct elements of $W_4\cup W_3$, we still get
			Similarly to the preceding cases, one checks that
			\[d_1+d_2+d_3\leq 3+12+(p-4-3)=p+8.\]
		\end{enumerate}
		\item $a\not\in\{x_1,x_2,x_3\}$, i.e. $a\in U$. %Here $y$ is adjacent to all three of $x_1,x_2,x_3$.
		By planarity, in this scenario at most one element $z_1$ of $U$ may be adjacent to all three of $x_1,x_2,x_3$.
		\begin{enumerate}
			\item Let $z_1\in U\setminus\{a\}$ be adjacent to all three of $x_1,x_2,x_3$. Then at most three of $U\setminus\{z_1\}$ (including $a$) may be adjacent to two of $x_1,x_2,x_3$. Moreover, in this case planarity does not allow for $x_1,x_2,x_3$ to be all pairwise adjacent. We compute
			\[d_1+d_2+d_3\leq 3+4+|\{z_1x_1,z_1x_2,z_1x_3\}|+6+(p-4-4)=p+8.\]
			\item
			Suppose that no element of $U$ is adjacent to all three of $x_1,x_2,x_3$. As $G-y-a$ is outerplanar, at most four of $U$ (including $a$) may be adjacent to two of $x_1,x_2,x_3$, so that
			\[d_1+d_2+d_3\leq 3+6+8+(p-4-4)=p+9.\]
		\end{enumerate}
	\end{enumerate}
\end{proof}
For $1\leq i<j\leq 3$, let
\[X_{ij}:=|\{\text{elements of } U\setminus\{a\} \text{ adjacent to both } x_i,x_j\}|.\]
By Lemma \ref{lem:a}, $X_{ij}\leq 1$ for every $i<j$. Therefore,
\[d_1+d_2+d_3\leq 3+6+|\{ax_1,ax_2,ax_3\}|+6+(p-4-4)=p+10.\]
Combining with \eqref{ten}, we actually have $X_{ij}=1$ for every $i<j$, and moreover each vertex of $U$ is adjacent to at least one of $x_1,x_2,x_3$. An illustration of $G$ is given in Figure \ref{fig:nu}. At least two of $d_1,d_2,d_3$ have to be equal to ensure that $\nu_m$ is unigraphic. Note that $G$ is maximal planar. Moreover, $a$ does not lie on the cycle $C$, $\deg(a)=3$, and all edges between two elements of $U$ lie on the cycle $C$. The proof of Theorem \ref{thm3} is complete.

\begin{remark}
Removing the condition $d_1,d_2,d_3\geq 7$ from Theorem \ref{thm3}, apart from $\nu_m(p)$, one also recovers the two exceptional sequences in \eqref{eqn:exc}.
\end{remark}

\begin{figure}[h!]
	\centering
	\tikzset{every picture/.style={line width=0.75pt}} %set default line width to 0.75pt        
	
	\begin{tikzpicture}[x=0.75pt,y=0.75pt,yscale=-0.7,xscale=0.7]
	%uncomment if require: \path (0,381); %set diagram left start at 0, and has height of 381
	
	%Straight Lines [id:da585231346046301] 
	\draw    (337,27) -- (457,227) ;
	\draw [shift={(457,227)}, rotate = 59.04] [color={rgb, 255:red, 0; green, 0; blue, 0 }  ][fill={rgb, 255:red, 0; green, 0; blue, 0 }  ][line width=0.75]      (0, 0) circle [x radius= 3.35, y radius= 3.35]   ;
	\draw [shift={(337,27)}, rotate = 59.04] [color={rgb, 255:red, 0; green, 0; blue, 0 }  ][fill={rgb, 255:red, 0; green, 0; blue, 0 }  ][line width=0.75]      (0, 0) circle [x radius= 3.35, y radius= 3.35]   ;
	%Straight Lines [id:da6457237823095701] 
	\draw    (217,227) -- (457,227) ;
	%Straight Lines [id:da08417249479145061] 
	\draw    (337,27) -- (217,227) ;
	\draw [shift={(217,227)}, rotate = 120.96] [color={rgb, 255:red, 0; green, 0; blue, 0 }  ][fill={rgb, 255:red, 0; green, 0; blue, 0 }  ][line width=0.75]      (0, 0) circle [x radius= 3.35, y radius= 3.35]   ;
	\draw [shift={(337,27)}, rotate = 120.96] [color={rgb, 255:red, 0; green, 0; blue, 0 }  ][fill={rgb, 255:red, 0; green, 0; blue, 0 }  ][line width=0.75]      (0, 0) circle [x radius= 3.35, y radius= 3.35]   ;
	%Curve Lines [id:da02051848975545001] 
	\draw    (217,227) .. controls (109.8,157) and (226.8,-35) .. (337,27) ;
	%Curve Lines [id:da9622059254740145] 
	\draw    (217,227) .. controls (215.8,354) and (455.8,354) .. (457,227) ;
	%Curve Lines [id:da6981951735550778] 
	\draw    (337,27) .. controls (428.8,-21) and (570.8,145) .. (457,227) ;
	%Straight Lines [id:da3323499007699404] 
	\draw    (337,157) -- (457,227) ;
	\draw [shift={(337,157)}, rotate = 30.26] [color={rgb, 255:red, 0; green, 0; blue, 0 }  ][fill={rgb, 255:red, 0; green, 0; blue, 0 }  ][line width=0.75]      (0, 0) circle [x radius= 3.35, y radius= 3.35]   ;
	%Straight Lines [id:da2533237727102775] 
	\draw    (337,27) -- (337,157) ;
	%Straight Lines [id:da09615970907888305] 
	\draw    (337,157) -- (217,227) ;
	%Straight Lines [id:da05568952739702837] 
	\draw    (337,27) -- (224.8,43) ;
	\draw [shift={(224.8,43)}, rotate = 171.88] [color={rgb, 255:red, 0; green, 0; blue, 0 }  ][fill={rgb, 255:red, 0; green, 0; blue, 0 }  ][line width=0.75]      (0, 0) circle [x radius= 3.35, y radius= 3.35]   ;
	\draw [shift={(337,27)}, rotate = 171.88] [color={rgb, 255:red, 0; green, 0; blue, 0 }  ][fill={rgb, 255:red, 0; green, 0; blue, 0 }  ][line width=0.75]      (0, 0) circle [x radius= 3.35, y radius= 3.35]   ;
	%Straight Lines [id:da07030182906124693] 
	\draw    (337,27) -- (190.8,86) ;
	\draw [shift={(190.8,86)}, rotate = 158.02] [color={rgb, 255:red, 0; green, 0; blue, 0 }  ][fill={rgb, 255:red, 0; green, 0; blue, 0 }  ][line width=0.75]      (0, 0) circle [x radius= 3.35, y radius= 3.35]   ;
	%Straight Lines [id:da27506195077197604] 
	\draw    (337,27) -- (174.8,139) ;
	\draw [shift={(174.8,139)}, rotate = 145.37] [color={rgb, 255:red, 0; green, 0; blue, 0 }  ][fill={rgb, 255:red, 0; green, 0; blue, 0 }  ][line width=0.75]      (0, 0) circle [x radius= 3.35, y radius= 3.35]   ;
	%Straight Lines [id:da2513663567780098] 
	\draw    (266,19) ;
	\draw [shift={(266,19)}, rotate = 0] [color={rgb, 255:red, 0; green, 0; blue, 0 }  ][fill={rgb, 255:red, 0; green, 0; blue, 0 }  ][line width=0.75]      (0, 0) circle [x radius= 3.35, y radius= 3.35]   ;
	%Straight Lines [id:da7751669276320969] 
	\draw    (217,227) -- (287.8,315) ;
	\draw [shift={(287.8,315)}, rotate = 51.18] [color={rgb, 255:red, 0; green, 0; blue, 0 }  ][fill={rgb, 255:red, 0; green, 0; blue, 0 }  ][line width=0.75]      (0, 0) circle [x radius= 3.35, y radius= 3.35]   ;
	%Straight Lines [id:da02866410991069257] 
	\draw    (217,227) -- (343.8,323) ;
	\draw [shift={(343.8,323)}, rotate = 37.13] [color={rgb, 255:red, 0; green, 0; blue, 0 }  ][fill={rgb, 255:red, 0; green, 0; blue, 0 }  ][line width=0.75]      (0, 0) circle [x radius= 3.35, y radius= 3.35]   ;
	%Straight Lines [id:da9686818402993309] 
	\draw    (217,227) -- (422.8,297) ;
	\draw [shift={(422.8,297)}, rotate = 18.79] [color={rgb, 255:red, 0; green, 0; blue, 0 }  ][fill={rgb, 255:red, 0; green, 0; blue, 0 }  ][line width=0.75]      (0, 0) circle [x radius= 3.35, y radius= 3.35]   ;
	%Straight Lines [id:da8097626441245858] 
	\draw    (239,287) ;
	\draw [shift={(239,287)}, rotate = 0] [color={rgb, 255:red, 0; green, 0; blue, 0 }  ][fill={rgb, 255:red, 0; green, 0; blue, 0 }  ][line width=0.75]      (0, 0) circle [x radius= 3.35, y radius= 3.35]   ;
	%Straight Lines [id:da7425598029518772] 
	\draw    (466.8,69) -- (457,227) ;
	\draw [shift={(457,227)}, rotate = 93.55] [color={rgb, 255:red, 0; green, 0; blue, 0 }  ][fill={rgb, 255:red, 0; green, 0; blue, 0 }  ][line width=0.75]      (0, 0) circle [x radius= 3.35, y radius= 3.35]   ;
	\draw [shift={(466.8,69)}, rotate = 93.55] [color={rgb, 255:red, 0; green, 0; blue, 0 }  ][fill={rgb, 255:red, 0; green, 0; blue, 0 }  ][line width=0.75]      (0, 0) circle [x radius= 3.35, y radius= 3.35]   ;
	%Straight Lines [id:da7315433098205153] 
	\draw    (500,161) ;
	\draw [shift={(500,161)}, rotate = 0] [color={rgb, 255:red, 0; green, 0; blue, 0 }  ][fill={rgb, 255:red, 0; green, 0; blue, 0 }  ][line width=0.75]      (0, 0) circle [x radius= 3.35, y radius= 3.35]   ;
	
	% Text Node
	\draw (318,139) node [anchor=north west][inner sep=0.75pt]   [align=left] {$\displaystyle a$};
	% Text Node
	\draw (328,-4) node [anchor=north west][inner sep=0.75pt]   [align=left] {$\displaystyle d_{1}$};
	% Text Node
	\draw (188,226) node [anchor=north west][inner sep=0.75pt]   [align=left] {$\displaystyle d_{2}$};
	% Text Node
	\draw (468,219) node [anchor=north west][inner sep=0.75pt]   [align=left] {$\displaystyle d_{3}$};
	% Text Node
	\draw (100,73) node [anchor=north west][inner sep=0.75pt]   [align=left] {$\displaystyle G-y:$};

	\end{tikzpicture}
	\caption{The realisation of the sequence $\nu_3(15)$.}
	\label{fig:nu}
\end{figure}
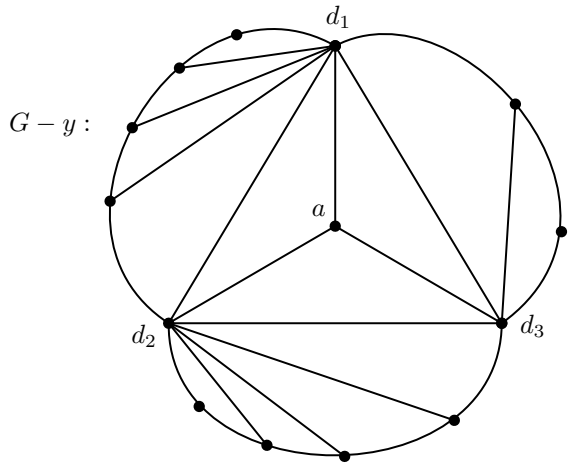

%\clearpage
\appendix
\section{Another family of unigraphic sequences: proof of Proposition \ref{twoseq}}
\label{appa}

%\subsection{First sequence}
Let $p\geq 16$ be even, and
\[\mu=p-2,\frac{p-2}{2}, 6^{(p-2)/2},3^{(p-2)/2}.\]
If $\mu$ has a polyhedral realisation $G$, then $G$ is a maximal planar graph (triangulation of the sphere).

As previously, call $y$ the vertex of degree $p-2$, $F:=G-y$, $a$ the only vertex of $G$ not adjacent to $y$, and $W:=V(F)\setminus\{a\}$. Thanks to the first statement in Lemma \ref{bigcycle}, every vertex of $F$, except possibly $a$, lies on a cycle $C$. Moreover, $a$ does not lie on $C$, as otherwise $G$ would have a quadrilateral face
\[[a,v_1,y,v_2].\]

We write
\[U=U_6\cup U_3,\]
where the elements of $U_i$ have degree $i$ in $G$. Note that $\deg_F(u)=\deg_G(u)-3$ for all $u\in U$. Call $R=F-E(C)$. There are three cases to consider.
\begin{enumerate}
\item
$\deg(a)=\frac{p-2}{2}$. In this case, we have
\[au\in E(G)\]
for all $u\in U_6$. The vertices of $U_3$ are isolated in $R$. Each $u\in U_6$ has two more edges that are not yet accounted for. Moreover, each $u\in U_6$ is cut off by planarity via the edges $au$, $u\in U_6$ from all others elements of $U_6$ save two. Then there is only one way to complete the construction of $G$ (Figure \ref{fig:mu}).

\begin{figure}[h!]
\centering\tikzset{every picture/.style={line width=0.75pt}} %set default line width to 0.75pt        

\begin{tikzpicture}[x=0.75pt,y=0.75pt,yscale=-0.7,xscale=0.56]
%uncomment if require: \path (0,381); %set diagram left start at 0, and has height of 381

%Straight Lines [id:da5392140622317678] 
\draw    (290,50) -- (360,50) ;
\draw [shift={(360,50)}, rotate = 0] [color={rgb, 255:red, 0; green, 0; blue, 0 }  ][fill={rgb, 255:red, 0; green, 0; blue, 0 }  ][line width=0.75]      (0, 0) circle [x radius= 3.35, y radius= 3.35]   ;
\draw [shift={(290,50)}, rotate = 0] [color={rgb, 255:red, 0; green, 0; blue, 0 }  ][fill={rgb, 255:red, 0; green, 0; blue, 0 }  ][line width=0.75]      (0, 0) circle [x radius= 3.35, y radius= 3.35]   ;
%Straight Lines [id:da29442499247469844] 
\draw    (360,50) -- (440,70) ;
\draw [shift={(440,70)}, rotate = 14.04] [color={rgb, 255:red, 0; green, 0; blue, 0 }  ][fill={rgb, 255:red, 0; green, 0; blue, 0 }  ][line width=0.75]      (0, 0) circle [x radius= 3.35, y radius= 3.35]   ;
%Straight Lines [id:da05171203330148644] 
\draw    (440,70) -- (500,130) ;
\draw [shift={(500,130)}, rotate = 45] [color={rgb, 255:red, 0; green, 0; blue, 0 }  ][fill={rgb, 255:red, 0; green, 0; blue, 0 }  ][line width=0.75]      (0, 0) circle [x radius= 3.35, y radius= 3.35]   ;
%Straight Lines [id:da34781268918404584] 
\draw    (500,130) -- (520,200) ;
\draw [shift={(520,200)}, rotate = 74.05] [color={rgb, 255:red, 0; green, 0; blue, 0 }  ][fill={rgb, 255:red, 0; green, 0; blue, 0 }  ][line width=0.75]      (0, 0) circle [x radius= 3.35, y radius= 3.35]   ;
%Straight Lines [id:da33304494036901433] 
\draw    (290,50) -- (210,70) ;
\draw [shift={(210,70)}, rotate = 165.96] [color={rgb, 255:red, 0; green, 0; blue, 0 }  ][fill={rgb, 255:red, 0; green, 0; blue, 0 }  ][line width=0.75]      (0, 0) circle [x radius= 3.35, y radius= 3.35]   ;
%Straight Lines [id:da647534877623392] 
\draw    (210,70) -- (150,130) ;
\draw [shift={(150,130)}, rotate = 135] [color={rgb, 255:red, 0; green, 0; blue, 0 }  ][fill={rgb, 255:red, 0; green, 0; blue, 0 }  ][line width=0.75]      (0, 0) circle [x radius= 3.35, y radius= 3.35]   ;
%Straight Lines [id:da33478887621179787] 
\draw    (150,130) -- (130,200) ;
\draw [shift={(130,200)}, rotate = 105.95] [color={rgb, 255:red, 0; green, 0; blue, 0 }  ][fill={rgb, 255:red, 0; green, 0; blue, 0 }  ][line width=0.75]      (0, 0) circle [x radius= 3.35, y radius= 3.35]   ;
%Straight Lines [id:da09430205072934439] 
\draw    (290,340) -- (360,340) ;
\draw [shift={(360,340)}, rotate = 0] [color={rgb, 255:red, 0; green, 0; blue, 0 }  ][fill={rgb, 255:red, 0; green, 0; blue, 0 }  ][line width=0.75]      (0, 0) circle [x radius= 3.35, y radius= 3.35]   ;
\draw [shift={(290,340)}, rotate = 0] [color={rgb, 255:red, 0; green, 0; blue, 0 }  ][fill={rgb, 255:red, 0; green, 0; blue, 0 }  ][line width=0.75]      (0, 0) circle [x radius= 3.35, y radius= 3.35]   ;
%Straight Lines [id:da6157335154440757] 
\draw    (360,340) -- (440,321.33) ;
\draw [shift={(440,321.33)}, rotate = 346.87] [color={rgb, 255:red, 0; green, 0; blue, 0 }  ][fill={rgb, 255:red, 0; green, 0; blue, 0 }  ][line width=0.75]      (0, 0) circle [x radius= 3.35, y radius= 3.35]   ;
%Straight Lines [id:da47515284835984706] 
\draw    (440,321.33) -- (500,265.33) ;
\draw [shift={(500,265.33)}, rotate = 316.97] [color={rgb, 255:red, 0; green, 0; blue, 0 }  ][fill={rgb, 255:red, 0; green, 0; blue, 0 }  ][line width=0.75]      (0, 0) circle [x radius= 3.35, y radius= 3.35]   ;
%Straight Lines [id:da2595452845322832] 
\draw    (500,265.33) -- (520,200) ;
\draw [shift={(520,200)}, rotate = 287.02] [color={rgb, 255:red, 0; green, 0; blue, 0 }  ][fill={rgb, 255:red, 0; green, 0; blue, 0 }  ][line width=0.75]      (0, 0) circle [x radius= 3.35, y radius= 3.35]   ;
%Straight Lines [id:da9819450083645107] 
\draw    (290,340) -- (210,321.33) ;
\draw [shift={(210,321.33)}, rotate = 193.13] [color={rgb, 255:red, 0; green, 0; blue, 0 }  ][fill={rgb, 255:red, 0; green, 0; blue, 0 }  ][line width=0.75]      (0, 0) circle [x radius= 3.35, y radius= 3.35]   ;
%Straight Lines [id:da14562626012536306] 
\draw    (210,321.33) -- (150,265.33) ;
\draw [shift={(150,265.33)}, rotate = 223.03] [color={rgb, 255:red, 0; green, 0; blue, 0 }  ][fill={rgb, 255:red, 0; green, 0; blue, 0 }  ][line width=0.75]      (0, 0) circle [x radius= 3.35, y radius= 3.35]   ;
%Straight Lines [id:da6002977227509458] 
\draw    (150,265.33) -- (130,200) ;
\draw [shift={(130,200)}, rotate = 252.98] [color={rgb, 255:red, 0; green, 0; blue, 0 }  ][fill={rgb, 255:red, 0; green, 0; blue, 0 }  ][line width=0.75]      (0, 0) circle [x radius= 3.35, y radius= 3.35]   ;
%Straight Lines [id:da19742476320028368] 
\draw    (150,265.33) -- (150,130) ;
%Straight Lines [id:da4390007221774763] 
\draw    (150,130) -- (290,50) ;
%Straight Lines [id:da8844957526672363] 
\draw    (290,50) -- (440,70) ;
%Straight Lines [id:da06802107874706476] 
\draw    (440,70) -- (520,200) ;
%Straight Lines [id:da42765355562339713] 
\draw    (150,265.33) -- (290,340) ;
%Straight Lines [id:da7290832686867361] 
\draw    (290,340) -- (440,321.33) ;
%Straight Lines [id:da818170366134269] 
\draw    (520,200) -- (440,321.33) ;
%Straight Lines [id:da007916133982818874] 
\draw    (325,195) -- (520,200) ;
\draw [shift={(325,195)}, rotate = 1.47] [color={rgb, 255:red, 0; green, 0; blue, 0 }  ][fill={rgb, 255:red, 0; green, 0; blue, 0 }  ][line width=0.75]      (0, 0) circle [x radius= 3.35, y radius= 3.35]   ;
%Straight Lines [id:da6125552302030293] 
\draw    (440,70) -- (325,195) ;
%Straight Lines [id:da05942771789377166] 
\draw    (290,50) -- (325,195) ;
%Straight Lines [id:da2786004152598509] 
\draw    (150,130) -- (325,195) ;
%Straight Lines [id:da2832056992184371] 
\draw    (150,265.33) -- (325,195) ;
%Straight Lines [id:da465262893304649] 
\draw    (325,195) -- (440,321.33) ;
%Straight Lines [id:da37354672552763435] 
\draw    (325,195) -- (290,340) ;

\end{tikzpicture}
\caption{The graph $F=G-y$ when $\deg(a)=\frac{p-2}{2}$, and $p=16$.}
\label{fig:mu}
\end{figure}
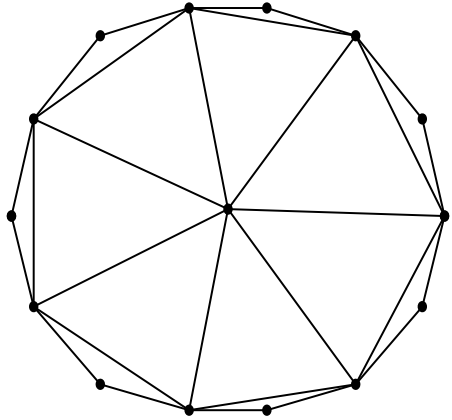

\item
$\deg(a)\neq\frac{p-2}{2}$. Let $x$ be the vertex of degree $\frac{p-2}{2}$ in $G$. Whether or not $ax\in E(G)$, $x$ is adjacent in $R$ to at least $\frac{p-2}{2}-4$ elements of $U_6$
\begin{equation}
\label{wi}
u_i, \quad 1\leq i\leq\frac{p-2}{2}-4,
\end{equation}
ordered in this way around $C$.
In particular, $u_1$ and $u_{\frac{p-2}{2}-4}$ are the closest among these to $x$ along $C$ on either side.

Consider the cycle $C'$ given by the edge $xu_1$ and the arc $xu_1$ of $C$ not containing $u_2$. Since $p\geq 16$, then $(p-2)/2-4\geq 3$, hence we can assume that $au_1\not\in E(G)$, and moreover $a$ does not lie inside or on $C'$. Therefore, this cycle delimits an outerplanar graph $F'$ (Figure \ref{fig:mu2}).

\begin{figure}[h!]
	\centering
	\includegraphics[width=6cm]{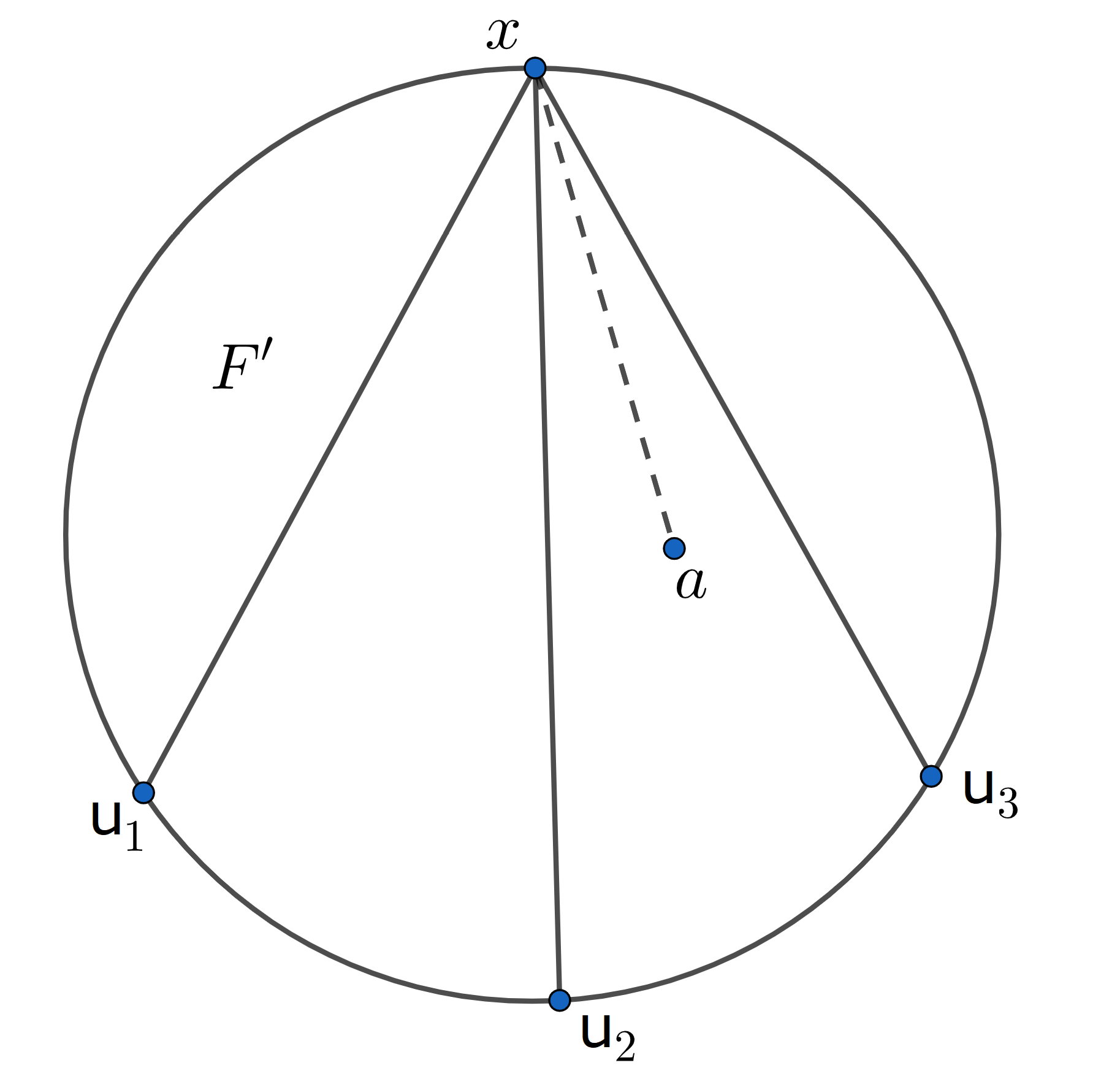}
	\caption{We sketch $F=G-y$ as usual. Whether or not $ax\in E(G)$, since $\deg_G(x)=7$, $x$ is adjacent to at least three elements of $U_6$ distinct from $a$. Then we can define the outerplanar graph $F'$, containing $u_1$ in the picture, such that $au_1\not\in E(G)$.}
	\label{fig:mu2}
\end{figure}

The graph $F'-x$ must have a non-trivial connected component $H$ containing $u_1$: otherwise, the degree in $G$ of $u_1$ would be at most $5$, seeing as $u_1$ is cut off by planarity from $a$ and all other $u_i$'s \eqref{wi} except $u_2$. By construction, $H$ contains, apart from $u_1$, a subset of vertices $U'\subset U_6$, of cardinality at most four, since these are distinct from the $u_i$'s in \eqref{wi}. The elements of $U'$ appear along the arc $xu_1$ of $C'$, and are hence cut off from $a$ and from all elements of $U_6$ except each other and $u_1$. They are not adjacent to $x$ either, by construction.

Now $\deg_H(u_1)$ is either $1$ or $2$, according to whether $u_1u_2\in E(G)$ or not. Each element of $U'$ has degree at least $3$ in $H$. Thereby, $|U'|\geq 4$, so that ultimately $|U'|=4$, i.e. $|V(H)|=5$. We use the labelling
\[V(H)=\{u'_1,u'_2,u'_3,u'_4,u_1\},\]
in this order around $C$. Suppose for contradiction that $u'_1u'_3\in E(G)$. Then $u'_2$ is cut off from $u'_4$ and $u_1$, and thus cannot have degree at least $3$ in $H$. Then $u'_1u'_3\not\in E(G)$, so that
\[u'_1u'_2,u'_1u'_4,u'_1u_1\in E(G).\]
But then $u'_3$ is cut off from $u_1$, and since $u'_1u'_3\not\in E(G)$, then $\deg_H(u'_3)\leq 2$. Either way, we have reached a contradiction.
\end{enumerate}
Having analysed all cases, we conclude that $\mu$ is unigraphic.

\end{document}